\documentclass[12pt]{amsart}

\usepackage{amsfonts,amssymb,graphicx,psfrag,fancyhdr}

\newtheorem{theorem}{Theorem}[section]

\newtheorem{prop}[theorem]{Proposition}
\newtheorem{defn}[theorem]{Definition}

\newtheorem{lemma}[theorem]{Lemma}
\newtheorem{rem}[theorem]{Remark}

\def\Id{\mathop{\rm Id}\nolimits}
\DeclareMathOperator{\diff}{\mathrm{Diff}}
\DeclareMathOperator{\Lip}{\mathrm{Lip}}
\DeclareMathOperator{\spec}{\mathrm{spec}}

\newcommand{\loc}{\mathrm{loc}}

\newcommand{\Z}{\mathbb{Z}}
\newcommand{\R}{\mathbb{R}}
\newcommand{\N}{\mathbb{N}}
\newcommand{\T}{\mathbb{T}}

\begin{document}

\title[ Liv\v{s}ic Theorems for
  Non-Commutative Groups]{Liv\v{s}ic Theorems for
  Non-Commutative Groups including Diffeomorphism Groups and Results
  on the Existence of Conformal Structures for Anosov Systems}

\author[R. de la Llave \and A. Windsor]{Rafael de la Llave \and
  Alistair Windsor}

\begin{abstract}
  The celebrated Liv\v{s}ic theorem \cite{Livsic71} \cite{Livsic72},
  states that given $M$ a manifold, a Lie group $G$, a transitive
  Anosov diffeomorphism $f$ on $M$ and a H\"older function $\eta: M
  \mapsto G$ whose range is sufficiently close to the identity, it is
  sufficient for the existence of $\phi :M \mapsto G$ satisfying $
  \eta(x) = \phi( f(x)) \phi(x)^{-1}$ that a condition --- obviously
  necessary --- on the cocycle generated by $\eta$ restricted to
  periodic orbits is satisfied.

  In this paper we present a new proof of the main result. These
  methods allow us to treat cocycles taking values in the group of
  diffeomorphisms of a compact manifold.  This has applications to
  rigidity theory. 

  The localization procedure we develop can be applied to obtain some
  new results on the existence of conformal structures for Anosov
  systems.
\end{abstract}

\maketitle

\section{Introduction}

The goal of this paper is to give a unified presentation -- sometimes
involving sharper technical conclusions -- of the existence of
solutions to coboundary equations over Anosov systems. 

We will give precise definitions in Section
\ref{sec:LieGroupPreliminaries} but we anticipate that the main
concern will be whether, given an Anosov diffeomorphism on a manifold
$M$, and function $\eta: M \rightarrow G$, where $G$ is a group
(either a Lie group or a group of diffeomorphisms), there exists a
function $\phi: M \rightarrow G$ such that
\begin{equation}\label{preliminary}
\phi \circ f  = \eta \cdot \phi. 
\end{equation}
(We will also discuss the flow case, but we omit a preliminary
discussion of it).

In the standard terminology, if we can find a solution to equation
\eqref{preliminary} then we say that the cocycle generated by $\eta$
is a coboundary.  There are many other variations of this question.
For example, instead of taking $G$ to be a Lie group, it is possible
to consider $G$ to be a Banach algebra \cite{MR1612768} or 
a bundle map.  We will omit
other important variations, such as when $(M,f)$ is a subshift. We
will mention in Section \ref{sec:conformal} the situation when $\phi$
are conformal structures and $\eta$ is natural map induced by the
tangent map. The study of such Liv\v{s}ic theorems in more geometric
contexts seems fruitful and will be pursued in further papers.

Cocycles arise naturally in many situations. They are intrinsic to the
definitions of special flows and skew products. In the study of
dynamical systems, the chain rule indicates that the derivative is a
cocycle.  The coboundary equation is geometrically natural and hences
arises naturally in a number of situations.  In particular,
\eqref{preliminary} appears naturally in the linearization of more
complicated equations, for example, it appears in the linearization of
conjugacy equations.  Hence, cocycle equations are basic tools for the
rigidity program \cite{Zimmer, MR1448015,Survey}.  Cocycle equations
appear also in the study of the asymptotic growth properties of
dynamical systems.  Diffeomorphism valued cocycles appear when
considering the behavior of system relative to its behavior on a
factor. The study of \eqref{preliminary} with $M$ a shift space, appears
naturally in thermodynamic formalism when one tries to decide whether
two potentials give rise to the same Gibbs state \cite{MR0399421,
  Bowen}. 

Note that, when $f^n(p) = p$, the existence of a solution to
\eqref{preliminary} implies that
\begin{displaymath}
  \eta (f^{n-1} p) \cdots \eta(f p) \cdot \eta(p) = \Id. 
\end{displaymath}
If this necessary condition holds for all periodic points $p \in M$
then we say that the \emph{periodic orbit obstruction} vanishes.

It is natural to ask whether the converse is true. Namely, if given an
$\eta$ such that the periodic orbit obstruction vanishes, whether
there is a $\phi$ solving \eqref{preliminary}.  Another natural
question -- especially for applications to geometry -- is whether the
solutions of \eqref{preliminary} are regular.

In this paper, we will concentrate in the existence question, but
since we will also study the case when $G$ is a group of $C^r$
diffeomorphisms, some regularity considerations will come in.

The question of the existence of solutions to \eqref{preliminary} was
first studied by Liv\v{s}ic in \cite{Livsic71} and \cite{Livsic72},
when $f$ is a topologically transitive Anosov system, and in
\cite{Bowen} when $f$ was a subshift of finite type.  We will refer as
\emph{Liv\v{s}ic theorems} to theorems that guarantee the existence of
solutions of \eqref{preliminary} under the hypothesis that $f$ is a
topologically transitive Anosov system or flow.\footnote{ Some
  references use also the spelling Livshitz. We prefer to maintain the
  spelling used in the papers by the author and in much of the
  subsequent literature.}  These papers showed that when $f$ is
transitive and $\eta$ is H\"older, then the periodic obstruction is
sufficient for the existence of a H\"older $\phi$.  (Continuity of
$\eta$ is definitely not enough and there are counterexamples).

There is a considerable literature on Liv\v{s}ic theorems in various
contexts.  \cite{MR47:5902} shows that for real-valued H\"older
cocycles the existence of $L^\infty$ solution to the coboundary
equation with a H\"older $\eta$ implies the existence of a H\"older
solution - in the literature this is often called the measurable
Liv\v{s}ic theorem.  We note that the main difficulty of measurable
Liv\v{s}ic theorems is that \eqref{preliminary} is not assumed to hold
everywhere, but only on a set of full measure. Hence restricting to a
periodic orbit -- or to the stable manifold of a periodic point do not
make sense. The interpretation of the periodic orbit obstruction is
far from obvious.

This was extended in the non-commutative case 
to certain $L^p$ spaces using Sobolev regularity
techniques in \cite{MR1849605} though interestingly the case of $L^1$
solutions remains open in the non-commutative case. A version of the
measurable Liv\v{s}ic theorem for cocycles taking values in
semi-simple Lie groups without any integrability assumptions appears
in \cite{MR1855844} though they need to assume additional bunching
conditions on the cocycle. Similarly a version of the measurable
Liv\v{s}ic theorem for cocycles taking values in compact Lie groups
appears in \cite{MR1489146}. This was extended, under some additional
hypotheses, to the case of cocycles taking values in connected Lie
groups \cite{MR1828477}.  Similar results for hyperbolic flows appear
in \cite{MR1637106}. That some additional hypotheses are necessary is
proved by a counterexample with a cocycle taking values in a solvable
Lie group in \cite{MR1745387}.  Extensions of the measurable
Liv\v{s}ic theorem to more general dynamics appear in
\cite{MR1621702,MR2153466}. Analogues for Markov maps or for systems
with discontinuities, appear in \cite{MR2032492,MR2166668,MR2183291}.
The study of the equation for skew-products appears in
\cite{MR2276106,MR2153466}. Cohomology equations over higher
dimensional actions were studied in
\cite{NiticaT03,NiticaT02,FisherM}. In general the vanishing of the
periodic orbit obstruction can be difficult to verify though in some
cases it is implied by spectral data~\cite{MR0405514}. If we can
verify that the periodic orbit obstruction vanishes on periodic orbits
of period less than $T$ for some $T$ then we may still obtain
approximate solutions to the cohomology equation~\cite{MR1062764}.

When $M$ is a quotient of a group and a lattice and $f$ is an
automorphism, the equation \eqref{preliminary} can be studied using
group representation techniques.  For example, \cite{Livsic72}
considered the case $M = \T^2 = \R^2/\Z^2$ and \cite{ColletEG84}
considered $M = PSL(2, \R)/\Gamma$ and $f$ a geodesic flow.  A more
general study of \eqref{preliminary} using representation techniques
is in \cite{Moore}.  The representation theory methods yield
information on the regularity questions also and the first results on
regularity appeared in \cite{Livsic72}. The representation theory
methods need to assume that $f$ has an algebraic structure, but not
that it Anosov \cite{Veech,FlaminioF}.  Of course, the representation
theory methods also lead to obstructions.  Though the representation
theory obstructions must be equivalent to the periodic orbit
obstructions, the connection is mysterious. 

For the particular case of geodesic flows further geometric
information on the solutions of the coboundary equation is obtained
in~\cite{GuilleminK} for surfaces, and in~\cite{GuilleminK3} for
$n$-dimensional manifolds with a pinching condition.  The method of
these papers uses harmonic analysis in some directions of the problem
and obtains, not only regularity, but also several other geometric
properties of the solutions (e.g. that they are polynomial in the
angle variables).

The regularity theory for the case when $G$ is commutative and $f$ is
any Anosov system appeared in \cite{MR840722}, the technique did not
use any representation theory. The main idea was to show that the
solutions are regular along the stable/unstable leaves of an Anosov
system and then, show that this implies regularity. There are a number
of approaches for obtaining regularity of the solution from the
regularity of the solution along transverse foliations.  Besides the
original one using elliptic regularity theory, we can mention Fourier
series \cite{HurderK90}, Morrey-Campanato spaces \cite{Journe88},
Sobolev embedding \cite{MR2093313} and Whitney regularity
\cite{MR1194019, NicolT07}.  Higher regularity for the non-commutative
case was studied very thoroughly in \cite{NiticaT98}.

In the case of cocycles taking values in a commutative group, two
cocycles are cohomologous if and only if their difference is a
coboundary. This does not extend to non-commutative groups and thus in
this case it is natural to ask whether there is a criteria on periodic
orbits to determine whether two given cocycles are cohomologous.  This
has been addressed in \cite{MR1695916, MR1695917,NiticaT98}.

We will discuss the existence of solutions in the context of cocycles
taking values in Lie groups and cocycles taking values in
diffeomorphism groups. Our result on cocycles taking values in a
diffeomorphism group extends the earlier results of \cite{NiticaT95}
on $\diff^r(\T^n)$ to $\diff^r(N)$ for any compact manifold $N$. A
different approach for higher rank actions appears in
\cite{MR2318540}. 

The proof we present for the finite dimensional case is not very
different from the proof of \cite{Livsic72b}, but we rearrange some of
the terms in the cancellations in a slightly different way so that as
many of the terms are geometrically natural -- only objects in the
same fiber of the tangent bundle are compared. We make sure that the
only comparisons which are not geometrically natural happen only in
points which are very close.  Our presentation clearly illustrates the
r\^ole played by localization assumptions.  These localizations
assumptions depend on the nature of group. They are always implied by
$\eta$ taking values in a small enough neighborhood of the identity,
but they are also automatic if the group is commutative, compact or
nilpotent.  The behavior of non-commutative cocycle equations in the
absence of such localization assumptions depends on the global
geometry of the group and remains an open problem.  In particular one
of Liv\v{s}ic's original theorems \cite[Theorem 3]{Livsic72} is not
justified. Resolving whether localization is necessary was posed as an
open problem by Katok during the Clay Mathematics Institute and MSRI
Conference on Recent Progress in Dynamics, 2004.

The rearrangement of the terms so that they are geometrically natural
is not crucial in the case of Lie groups -- there are many other
alternative rearrangements which work -- but it becomes important in
the case that the group $G$ is a diffeomorphism group. In the
pioneering work \cite{NiticaT95}, the authors needed to assume that
the manifolds were essentially flat.  In Section
\ref{sec:diffeomorphismgroups}, we remove this assumption. The paper
\cite{NiticaT95} also contains applications of the results on
cohomology equations to rigidity of partially hyperbolic actions.  If
one inserts the improvements presented here on the arguments of the
the argument in \cite{NiticaT95,NiticaT01} one can also extend the
results of those papers.

\section{Some preliminaries on Anosov systems}

\subsection{Definitions}

Our cocycles will be over Anosov diffeomorphisms and Anosov
flows. These exhibit the strongest form of hyperbolicity, namely uniform
hyperbolicity on the entire manifold. 

\begin{defn}[Anosov Diffeomorphism]
  Let $M$ be a compact Riemannian manifold. A diffeomorphism $f \in
  \diff^r(M)$ for $r \geq 1$ is called an Anosov diffeomorphism if
  there exist $C>0$ and $\lambda <1$ and a splitting of the tangent
  bundle
  \begin{equation*}
    T M = E^s \oplus E^u
  \end{equation*}
  such that
  \begin{enumerate}
  \item For all $v \in E^s_x$ and for all $n > 0$
    \begin{equation*}
      \|Df^n_x v \| < C \lambda^n \| v \|.
    \end{equation*}

  \item For all $v \in E^u_x$ and for all $n < 0$
    \begin{equation*}
       \|Df^n_x v \| < C \lambda^{|n|} \| v \|.
    \end{equation*}
  \end{enumerate}
  If $f$ is an Anosov diffeomorphism with constants $C>0$ and
  $\lambda<1$ then we will call $f$ $\lambda$-hyperbolic.
\end{defn}

\begin{rem}
  Note that in the definition of Anosov diffeomorphism the metric
  enters explicitly. For a compact manifold $M$, if a diffeomorphism
  is Anosov in one metric, then it is Anosov in all metrics, and one
  can even take the same $\lambda$ for all the metrics.  The constant
  $C$, however, depends both on the metric and on the $\lambda$ that
  we choose. If $f$ is $\lambda$-hyperbolic then it is possible to
  choose a metric, as smooth as $M$, such that $f$ is
  $\lambda'$-hyperbolic with constant $C=1$ for any $\lambda'$ with
  $\lambda < \lambda' < 1$. Furthermore the metric may be chosen such
  that the sub-bundles $E^s$ and $E^u$ are orthogonal. Such a metric
  is sometimes called an ``adapted metric'' \cite{Mather68}.
\end{rem}

\begin{defn}[Anosov Flow]
  Let $M$ be a compact Riemannian manifold. A flow $f^t: M \rightarrow
  M$ is called an Anosov flow if there exist $C>0$ and $\lambda >0 $
  and a splitting of the tangent bundle
  \begin{equation*}
    T M = E^s \oplus E^0 \oplus E^u
  \end{equation*}
  such that
  \begin{enumerate}

  \item At each $x \in M$ the subspace $E^0_x$ is one dimensional and
    \begin{equation*}
    \bigl.\frac{d}{dt} f^t(x) \bigr|_{t=0} \in E^0_x\setminus\{0\}.
  \end{equation*}

  \item For all $v \in E^s_x$ and for all $t > 0$
    \begin{equation*}
      \|Df^t_x v \| < C e^{-\lambda t} \| v \|.
    \end{equation*}

  \item For all $v \in E^u_x$ and for all $t < 0$
    \begin{equation*}
       \|Df^t_x v \| < C e^{-\lambda |t|} \| v \|.
    \end{equation*}
  \end{enumerate}
  If $f^t$ is an Anosov flow with constants $C>0$ and $\lambda>0$ then
  we will call $f^t$ $\lambda$-hyperbolic.
\end{defn}

\subsection{Anosov Foliations}
\label{sec:foliations}

The sub-bundles $E^s, E^u \subset TM$ from the definition of Anosov
diffeomorphisms and flows are called the stable and unstable bundles
respectively. There are foliations $W^s$ and $W^u$ associated to $E^s$
and $E^u$ such that $T_xW^s(x) = E^s_x$ and $T_xW^u(x) = E^u_x$. These
foliations can be characterized by:
\begin{equation}
  \label{foliationdef}
  \begin{split} 
    W^s(x) 
    &= \{ y \in M : d_M(f^n(x), f^n(y)) \to 0 \} \\ 
    &= \{ y \in M : d_M(f^n(x), f^n(y)) \leq C_{x,y} \lambda^n, \; n >
    0 \}\\ 
    W^u(x)
    &=\{ y \in M : d_M(f^n(x), f^n(y)) \to 0 \\ 
    &=\{ y \in M : d_M(f^n(x), f^n(y)) \le C_{x,y} \lambda^{|n|}, \; n
    < 0 \}
  \end{split}
\end{equation}
Similarly for flows one may define the center stable and center
unstable bundles $E^{cs}$ and $E^{cu}$ by $E^{cs}_x=E^0_x \oplus
E^s_x$ and $E^{cu}_x = E^0_x \oplus E^u_x$. These are again integrable
and have associated foliations $W^{cs}$ and $W^{cu}$ respectively. 

The global structure of the stable and unstable manifolds may be quite
bad -- they are only immersed sub-manifolds. Moreover, though the
leaves of the foliation are as smooth as the map or flow the holonomy
between leaves is generally less regular than the map (the regularity
is limited by ratios of contraction exponents).  There are
many excellent sources for the theory of invariant manifolds -- see
for example \cite{HirschPS77} for an exposition of the Hadamard
approach and see \cite{MR2186242} for an exposition of the Perron
approach.  The original method of Poincar\'e was reexamined in modern
language and extended in \cite{CabreFL03}. A more comprehensive survey
is \cite{Pesin04}.

\subsection{The Anosov Closing Lemma}

For us the most crucial property of Anosov systems is the following
shadowing lemma, often called the Anosov closing lemma. 

\begin{lemma}[Anosov Closing Lemma for Flows]\label{AnosovClosingFlow}
  Let $f^t$ be an Anosov flow on a compact Riemannian manifold $M$.
  There exist $\epsilon_0 >0$, $K >0$, and $T_0 >0$ such that if for
  some $ T > T_0$
  \begin{equation*}
    d_M \bigl( f^Tx,x \bigr) < \epsilon_0
  \end{equation*}
  we can find a unique periodic point $p$ of period $T+\Delta$ satisfying 
  \begin{enumerate}
  \item[{\rm a)}]
    $\qquad \begin{aligned}[t]
      d_M(x,p) &\leq  K \, \epsilon_0\\
      d_M\bigl( f^T(x),p \bigr) &\leq  K \, \epsilon_0\\
      |\Delta|&\leq K \, \epsilon_0
    \end{aligned}$
  \item[{\rm b)}] $\qquad W_{\mathrm{loc}}^{s}(x) \cap W_{\mathrm{loc}}^{u}(p)
    \neq \varnothing$
  \end{enumerate}
  Moreover, this unique point satisfies: 
  \begin{enumerate}
  \item[{\rm c)}] $\qquad 
    \begin{aligned}[t]
      d_M(x,p) &\leq  K \,d_M \bigl( f^Tx,x \bigr)\\
      d_M\bigl( f^T(x),p \bigr) &\leq  K \, d_M \bigl( f^Tx,x \bigr)\\
      |\Delta|&\leq K \, d_M \bigl( f^Tx,x \bigr)
    \end{aligned}$
  \item[{\rm d)}] $\qquad W_{\mathrm{loc}}^{s}(x) \cap
    W_{\mathrm{loc}}^{u}(p) = \{z\}$
  \end{enumerate}
\end{lemma}

\begin{rem}
  The statement of Lemma \ref{AnosovClosingFlow} is more involved than
  the corresponding one for diffeomorphisms, Lemma
  \ref{AnosovClosingDiffeomorphism}, because all the points in a
  periodic orbit are periodic, so that, in the case of flows, the set
  of periodic points of a given period, that lie in a neighborhood, is
  not discrete.  We can hope for uniqueness of the periodic point $p$
  only if some additional condition such as b) is imposed.  This is
  not needed in the case of diffeomorphisms, since periodic points of
  a fixed period are isolated.  Similarly, in the case of
  diffeomorphisms, since the set of periods is discrete, we do not
  have to consider the $\Delta$ that changes the period.

  \begin{psfrags}   
    \psfrag{10}{$ z$}
    \psfrag{9}{$ x$}
    \psfrag{8}{$ f^Tz$}
    \psfrag{7}{$ f^Tx$}
    \psfrag{2}{$ W^u(f^Tx)$}
    \psfrag{6}{$ p=f^{T+\Delta}p$}
    \psfrag{5}{$ f^Tp$}
    \psfrag{1}{$ W^u(x)$}
    \psfrag{4}{$ W^s(p)$}
    \psfrag{3}{$ W^s( f^Tp)$}
    \begin{figure}
      \begin{center}
        \includegraphics{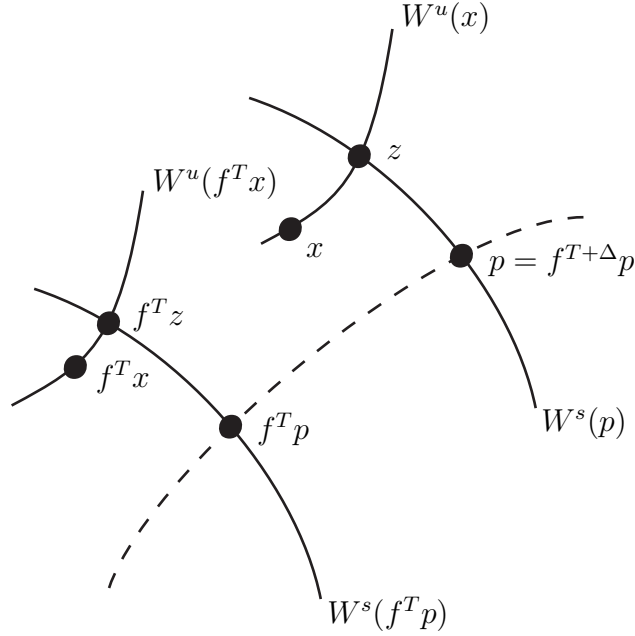}
        \label{shadow2}
        \caption{Illustration of the closing lemma, Lemma
          \ref{AnosovClosingFlow}}
      \end{center}
    \end{figure}
  \end{psfrags}
\end{rem}

\begin{lemma}[Anosov Closing Lemma for
  Diffeomorphisms]\label{AnosovClosingDiffeomorphism}
  Let $f$ be an Anosov diffeomorphism on a compact Riemannian manifold
  $M$.  There exist $\epsilon_0 >0$, $K >0$, and $\lambda >0$ such
  that if for some $n \in \Z$
  \begin{equation*}
    d_M \bigl( f^nx,x \bigr) < \epsilon_0
  \end{equation*}
  we can find a unique periodic point $p$ of period $n$ satisfying
  \begin{enumerate}
  \item[{\rm a)}]
    $\qquad \begin{aligned}[t]
      d_M(x,p) &\leq  K \, \epsilon_0\\
      d_M\bigl( f^nx,p \bigr) &\leq  K \, \epsilon_0\\
    \end{aligned}$
   \end{enumerate}
  Moreover, this unique point satisfies: 
  \begin{enumerate}
  \item[{\rm b)}] $\qquad 
    \begin{aligned}[t]
     d_M(x,p) &\leq  K \,d_M \bigl( f^n x,x \bigr)\\
      d_M\bigl( f^n x ,p \bigr) &\leq  K \, d_M \bigl( f^n x, x \bigr)\\
    \end{aligned}$
  \item[{\rm c)}] $\qquad W_{\mathrm{loc}}^{s}(x) \cap
    W_{\mathrm{loc}}^{u}(p) = \{z\}$
  \end{enumerate}
\end{lemma}

\subsection{Cocycles}

\begin{defn}
  Let $G$ be a group. A $G$-valued cocycle over a homeomorphism $f: M
  \rightarrow M$ is a map $\Phi:M \times \Z \rightarrow G$ that
  satisfies
  \begin{equation}
    \label{eq:CocycleConditionDiffeo}
    \Phi(x,m + n) = \Phi(f^n x, m) \cdot \Phi(x, n)
  \end{equation}
  for all $x \in M$, and $m,n \in \Z$. Here $\cdot$ denotes the group
  operation.
\end{defn}

\begin{defn}
  Let $G$ be a group. A $G$-valued cocycle over a flow $f^t: M
  \rightarrow M$ is a map $\Phi:M \times \R \rightarrow G$ that
  satisfies
  \begin{equation}
    \label{eq:CocycleConditionFlow}
    \Phi(x,s + t) = \Phi(f^t x, s) \cdot \Phi(x, t).
  \end{equation}
  where $x \in M$, and $s, t \in \R$. Here $\cdot$ denotes the group
  operation.
\end{defn}

\begin{rem}
  These are special cases of the more general definition of a cocycle
  over a group action. 
\end{rem}

Any cocycle $\Phi$ over a homeomorphism $f$ is determined entirely by its
generator $\eta: M \rightarrow G$ given by $\eta(x) = \Phi(x,1)$. The
cocycle $\Phi$ is reconstructed by
\begin{equation*}
  \Phi(x,n) =
  \begin{cases}
    \eta(f^{n-1}x) \cdots \eta(x) & n \geq 1\\
    \Id& n=0\\
    \eta^{-1}(f^{n}x) \cdots \eta^{-1}(f^{-1} x) & n \leq -1
  \end{cases}.
\end{equation*}

In the flow case the duality is not as complete. However if $G$ is a
Lie group with Lie algebra $\mathfrak{g}$, $f^t$ is a smooth flow on
$M$, and $\Phi$ is smooth then $\Phi$ is determined by its
infinitesimal generator $\eta: M \rightarrow \mathfrak{g}$ given by
\begin{equation*}
  \eta(x) = \frac{d}{dt} \Bigl. \Phi(x,t) \Bigr|_{t=0}.
\end{equation*}
The cocycle can the be reconstructed as the unique solution to 
\begin{equation*}
  \frac{d}{dt} \Phi(x,t) = DR_{\Phi(x,t)} \eta(f^t x), \qquad \Phi(x,0)=\Id
\end{equation*}
where $R_{\Phi(x,t)} : g \mapsto g \cdot \Phi(x,t)$ is the operation
of right multiplication by $\Phi(x,t)$, and hence $D R_{\Phi(x,t)}:
\mathfrak{g} \rightarrow T_{\Phi(x,t)} G$.

\section{Liv\v{s}ic Theory for Lie Group Valued Cocycles}
\label{sec:LieGroupPreliminaries}
Let $G$ be a Lie group endowed with a Riemannian metric. Let $d_G$ be
the length metric on the path-connected component of the identity in
$G$. If $G$ is non-compact the group operation need not be Lipshitz
but it is Lipshitz on any compact path-connected domain. As $G$ is a
Lie group the multiplication operator is smooth. Hence for every $g
\in G$ the operators $L_h: g \mapsto h \cdot g$ and $R_h: g \mapsto g
\cdot h$ are smooth.

For $F \in C^1(G,G)$ define
\begin{equation*}
  |F|_r := \max \{ \|D_g F \| : g \in \overline{B}(\Id, r) \}
\end{equation*}
This gives us the following estimates
\begin{equation}
  \label{eq:RightMultiplication}
  \text{ if $d_G\bigl( g h^{-1} , \Id \bigr)< r$ then } d_G\bigl( g, h
  \bigr) \leq |R_h|_r \, d_G\bigl( g h^{-1}, \Id \bigr).
\end{equation}
In general, without some localization, we cannot relate $d_G\bigl( g
h^{-1}, \Id \bigr)$ and $ d_G\bigl( g, h \bigr)$ but in the case of a
continuous function $\eta : M \rightarrow G$ on a compact manifold $M$
there exists $K > 0$ such that
\begin{equation}\label{eq:UniformEstimateEta}
  \begin{split}
  d_G\bigl( \eta(x) \eta^{-1}(y) , \Id \bigr) &< K \,
  d_G \bigl( \eta(x), \eta(y) \bigr)\\
  d_G\bigl( \eta^{-1}(x) \eta(y) , \Id \bigr) &< K \, 
  d_G \bigl( \eta(x), \eta(y) \bigr)
\end{split}
\end{equation}
since $\eta(M)$ is a compact subset of $G$.

\begin{theorem}\label{thm:LivsicLieGroupDiffeo}
  Let $M$ be a compact Riemannian manifold, $f:M \rightarrow M$ be a
  $C^1$ topologically transitive $\lambda$-hyperbolic Anosov
  diffeomorphism, and $G$ be a Lie group. Let $\Phi \in
  C^\alpha( M \times \Z , G )$ be a cocycle. For $x,y \in M$ define
  $\Delta_{x,y}^n:G \rightarrow G$ by
  \begin{equation*}
    \Delta_{x,y}^n(g) = \Phi^{-1}(x,n) \, g \, \Phi(y,n).
  \end{equation*}
  Suppose there exists $\rho > 1$ such that for all $x,y \in M$ and all
  $n \in \Z$
  \begin{equation}
    \label{eq:DiffeomorphismLocalization}
    | \Delta_{x,y}^n |_{\rho^{-|n|}} \leq K \rho^{|n|}.  
  \end{equation}
  Suppose that that for the the pair $(f,\Phi)$:
  \begin{enumerate}
  \item The periodic orbit obstruction vanishes:
    \begin{quote}
      If $f^n p = p$ then $\Phi(p,n) = \Id$.
    \end{quote}
    
  \item The hyperbolicity condition is satisfied:
    \begin{equation}
      \label{eq:HypCondDiffeo}
      \rho \lambda^\alpha <1 
    \end{equation}
  \end{enumerate} 
  then there exists $\phi \in C^\alpha(M,G)$ that solves
  \begin{equation}
    \label{eq:Coboundary}
    \Phi(x,n) = \phi(f^n x) \phi^{-1}(x). 
  \end{equation}
  Moreover, if $\hat\phi$ is any other continuous solution to
  \eqref{eq:Coboundary} then
  \begin{equation*}
    \hat\phi = \phi \cdot g
  \end{equation*}
  for some $g \in G$
\end{theorem}

\begin{rem}
  Condition \ref{eq:DiffeomorphismLocalization} will be examined in
  greater detail in Section \ref{sec:localization}. If the group is a
  commutative Lie matrix group endowed with a matrix norm, or a
  commutative Lie group endowed with an invariant metric then no
  localization condition is required. 
\end{rem}

\begin{proof}
  Rearranging \eqref{eq:Coboundary} we obtain
  \begin{equation}
    \label{eq:IteratedCoboundary}
    \phi(f^n x) = \Phi(x,n) \cdot \phi(x)
  \end{equation}
  Thus we see that fixing $\phi(x)$ immediately determines $\phi$ on
  the entire orbit of $x$. Since $f$ is topologically transitive there
  exists a point $x^*$ with a dense orbit, $\mathcal{O}(x^*)$. Fixing
  $\phi(x^*)$ therefore defines a function $\phi:\mathcal{O}(x^*)
  \rightarrow G$. This shows that any continuous solution $\hat\phi$
  to \eqref{eq:Coboundary} is uniquely determined by
  $\hat\phi(x^*)$. Thus choosing $g = \phi^{-1}(x^*)\cdot
  \hat\phi(x^*)$ we get
  \begin{equation*}
    \hat\phi(x) = \phi(x) \cdot g.
  \end{equation*}

  It remains to show that the $\phi:\mathcal{O}(x^*) \rightarrow G$
  defined by \eqref{eq:IteratedCoboundary} can be extended to a
  $C^\alpha$ function $\phi: M \rightarrow G$. Standard arguments show
  that we can extend $\phi:\mathcal{O}(x^*) \rightarrow G$ provided it
  is uniformly $C^\alpha$ on $\mathcal{O}(x^*)$, i.e. there exists a
  $\delta>0$ and $K>0$ such that
  \begin{multline}
    \label{eq:GlobalHolder}
    \text{ if $d_M( f^{n+N}x^*, f^nx^*)<\delta$ then}\\
    d_G\bigl(\phi(f^{n+N}x^*), \phi(f^nx^*)\bigr) < K d_M( f^{n+N}x^*,
    f^nx^*)^\alpha.
  \end{multline}

  We have the following basic estimate from \eqref{eq:RightMultiplication}
  \begin{multline}
    \label{eq:LocalBounds}
     \text{ if $d_G\bigl( \Phi(f^nx^*,N), \Id \bigr)< 1$ then}\\
     d_G\bigl( \phi(f^{n+N}x^*) ,  \phi(f^nx^*) \bigr)  \leq
     | R_{\phi(f^nx^*)} |_1 \,
    d_G\bigl( \Phi(f^nx^*,N), \Id \bigr).
  \end{multline}

  First we show that the following H\"older condition on $\Phi$, 
  \begin{multline}
    \label{eq:LocalHolder}
    \text{ if $d_M( f^{n+N}x^*, f^nx^*)<\delta$ then}\\
    d_G\bigl( \Phi(f^nx^*,N), \Id \bigr)  < K d_M( f^{n+N}x^*, f^nx^*)^\alpha
  \end{multline}
  is equivalent to the H\"older condition \eqref{eq:GlobalHolder}.
  
  Choose $0<\delta' \leq \delta$ so that $K (\delta')^\alpha<1$. The
  collection $\{B(f^n x^*, \delta') \}_{n \in \Z}$ is an open cover of
  $M$ and therefore by compactness we have a finite sub-cover
  $\{B(f^{n_i} x^*, \delta') \}_{i = 1}^m$. Let
  \begin{align*}
    L &= \max_{i=1, \dots, m} |R_{\phi(f^{n_i}x^*)} |_1 
    \intertext{and }
    O &= \max_{i=1, \dots, m} d_G \bigl(\phi(f^{n_i}x^*),\Id \bigr).
  \end{align*}
  Given an arbitrary $n \in \Z$ we choose $1 \leq i \leq m$ such that
  $d_M(f^n x^*, f^{n_i}x^*)< \delta'$. From \eqref{eq:LocalHolder} we
  get $d_G \bigl( \Phi(f^{n_i}x^*, n-n_i), \Id \bigr) < 1$ and hence
  can use \eqref{eq:RightMultiplication}
  \begin{align*}
    d_G\bigl( \phi(f^nx^*) &, \Id)\\  
    &\leq d_G \bigl( \phi(f^nx^*) ,
    \phi(f^{n_i}x^*) \bigr) + d_G \bigl( \phi(f^{n_i}x^*), \Id \bigr)\\
    &\leq |R_{\phi(f^{n_i}x^*)}|_1 \, d_G \bigl( \Phi(f^{n_i}x^*, n-n_i),
    \Id \bigr) +d_G \bigl( \phi(f^{n_i}x^*), \Id \bigr)\\
    &\leq O + L.
  \end{align*} 
  Now that we know that $\phi\bigl(\mathcal{O}(x^*)\bigr)$ is a
  precompact subset of $G$ we have
  \begin{equation*}
    L = \sup_{n \in \Z} |R_{\phi(f^nx^*)}|_1 < \infty.
  \end{equation*}
  and hence from \eqref{eq:LocalBounds}
  \begin{multline}
    \label{eq:GlobalBounds}
    \text{ if $d_G\bigl( \Phi(f^nx^*,N), \Id \bigr)< 1$ then}\\
    d_G\bigl( \phi(f^{n+N}x^*) ,  \phi(f^nx^*) \bigr)  \leq L \,
    d_G\bigl( \Phi(f^nx^*,N), \Id \bigr).
  \end{multline}
  Then applying \eqref{eq:LocalHolder} we obtain
  \eqref{eq:GlobalHolder}. Thus to prove the theorem it suffices to prove
  \eqref{eq:LocalHolder}.

  Applying the Anosov Closing Lemma, Lemma
  \ref{AnosovClosingDiffeomorphism}, we obtain $p$ with $f^Np=p$ and
  $d_M( f^n x^*, p) \leq K d_M( f^{n+N}x^*, f^nx^*)$ and $z \in W^s(p)
  \cap W^u(f^nx^*)$.

  We will compare the cocycle $\Phi$ along two trajectories that
  converge exponentially in forward time. Let $C_F(m) :=
  \Phi^{-1}(p,m) \cdot \Phi(z,m)$. We have for $m \in \N$.
  \begin{align*}
    C_F(m+1) &= \Phi^{-1}(p,m)  \eta^{-1}(f^mp) \cdot \eta (f^m z)
    \Phi(z,m)\\
    &= \Delta_{p,z}^m \bigl(   \eta^{-1}(f^mp) \cdot \eta (f^m z) \bigr) 
  \end{align*}
  We have for $d_M(f^{n+N}x^*, f^nx^*)$ sufficiently small
  \begin{align*}
    d_G\bigl( \eta^{-1}(f^mp) \cdot \eta (f^m z), \Id \bigr) & \leq C_1 \,
    d_G\bigl( \eta(f^m p) , \eta(f^mz) \bigr) &&
    \text{\eqref{eq:UniformEstimateEta}}\\ 
    & \leq C_2 \, d_M(f^m p , f^m z)^\alpha && \eta \in
    C^\alpha(M,G)\\
    & \leq C_3 \, \lambda^{\alpha m} \, d_M(p,z)^\alpha&&z \in W^s(p) \\
    & \leq C_4 \, \lambda^{\alpha m} \, d_M(f^{n+N}x^*, f^nx^*)^\alpha\\
    & \leq C_4 \, \lambda^{\alpha m} \, \delta^\alpha.
  \end{align*}
  so we can choose $\delta>0$ sufficiently small that for $m \geq 0$
   \begin{equation*}
     d_G\bigl( \eta^{-1}(f^mp) \cdot \eta (f^m z), \Id \bigr) <
     \lambda^{\alpha m} < \rho^{-m}
   \end{equation*}
   and hence we can apply \eqref{eq:DiffeomorphismLocalization} to obtain
   \begin{align*}
     d_G\bigl( C_F(m+1), C_F(m)\bigr) &\leq d_G \bigl( \Delta_{p,z}^m
     \bigr( \eta^{-1}(f^mp) \cdot \eta (f^m z) \bigr), \Delta_{p,z}^m \bigl(
     \Id \bigr) \bigr) \\
     &\leq | \Delta_{p,z}^m |_{\rho^{-m}} \, d_G\bigl( \eta^{-1}(f^mp)
     \cdot \eta (f^m z), \Id \bigr)
   \end{align*}
   Thus
   \begin{align*}
     d_G\bigl( C_F(m+1), C_F(m) \bigr) &\leq | \Delta_{p,z}^m
     |_{\rho^{-m}} \, C_4 \, \lambda^{\alpha m} \, d_M(f^{n+N}, f^nx^*)^\alpha\\
     & \leq C_5 \, (\rho \lambda^\alpha)^m \, d_M(f^{n+N}x^*,
     f^nx^*)^\alpha
  \end{align*}
  By using the hyperbolicity assumption \eqref{eq:HypCondDiffeo} we
  get the following bound
  \begin{equation*}
     d_G \bigl( C_F(m), \Id \bigr) \leq \frac{C_5}{1 - \rho
      \lambda^\alpha} \,  d_M( f^{n+N}x^*, f^nx^*)^\alpha.   
  \end{equation*}
  In particular, since $\Phi(p,N) = \Id$, we have $C_F(N) = \Phi(z,N)$
  and hence obtain
  \begin{equation}
    \label{eq:PhizEstimate}
    d_G \bigl( \Phi(z,N) , \Id \bigr)  \leq \frac{C_5}{1 - \rho
      \lambda^\alpha} \,  d_M( f^{n+N}x^*, f^nx^*)^\alpha. 
  \end{equation}
  This means that $ d_G \bigl( \Phi(z,N) , \Id \bigr)$ is uniformly
  bounded and consequently $|R_{\Phi(z,N)}|_1$ is uniformly bounded. 

  Now we compare $\Phi$ along two trajectories that converge
  exponentially in backwards time. For $m \in \N$ we define 
  \begin{align*}
    C_R(m)&= \Phi^{-1}( f^{n+N}x^*,-m) \cdot \Phi(f^Nz,-m)\\
    &= \Delta_{f^{n+N}x^*,f^N z}^{-m}(\Id)
  \end{align*}
  From the definition of the cocycle we obtain
  \begin{equation*}
    C_R(m+1) =  \Delta_{f^{n+N}x^*,f^Nz}^{-m}\bigl( \eta(f^{n+N-m-1}x^*)
    \eta^{-1}(f^{N-m-1}z) \bigr).
  \end{equation*}
  Exactly as before, we are able to estimate
  \begin{equation*}
    d_G \bigl( C_R(m+1), C_R(m) \bigr)  \leq C_5 \, (\rho \lambda^\alpha)^m
    \, d_M(f^{n+N}x^*, f^nx^*)^\alpha. 
  \end{equation*}
  By the hyperbolicity assumption we have $\rho \lambda^\alpha < 1$
  and hence,as before, we get the following bound, uniform in $m$
  \begin{equation*}
    d_G \bigl( C_R(m) , \Id \bigr) \leq \frac{ C_5}{1 -
      \rho\lambda^\alpha} d_M( f^{n+N}x^*, f^nx^*)^\alpha.  
  \end{equation*}
  In particular, for $m = N$
  \begin{equation}
    \label{eq:DNEstimate}
    d_G \bigl( C_R(N) , \Id \bigr) \leq
    \frac{C_5}{1 - \rho\lambda^\alpha} d_M(f^{n+N}x^*, f^nx^*)^\alpha. 
  \end{equation}
  From the cocycle property we obtain
  \begin{align*}
    \Phi(f^Nz,-N) &= \Phi^{-1}(z,N)
    \intertext{and} 
    \Phi^{-1}(f^{N+n}x^*,-N) &=\Phi(f^nx^*,N). 
  \end{align*}
  Thus
  \begin{align*}
    d_G \bigl( \Phi(f^nx^*,&N), \Id \bigr) \\
    &\leq
    \begin{aligned}[t]
      & d_G \bigl( \Phi^{-1}(
    f^{n+N}x^*,-N) \Phi(f^Nz,-N ) \Phi(z,N) , \Phi(z,N) \bigr) \\
    & \quad + d_G(\Phi(z,N) , \Id )
    \end{aligned}\\
    & \leq |R_{\Phi(z,N)}|_1 \, d_G\bigl(C_R(N),\Id \bigr)+ d_G\bigl(
    \Phi(z,N) ,\Id \bigr).
  \end{align*}
  From \eqref{eq:PhizEstimate}, the fact $|R_{\Phi(z,N)}|_1$ is
  uniformly bounded, and \eqref{eq:DNEstimate} we get
  \begin{equation*}
    d_G \bigl( \Phi(f^nx^*,N), \Id \bigr)  \leq L \, d_M( f^{n+N}x^*,
    f^nx^*)^\alpha 
  \end{equation*}
  which establishes \eqref{eq:LocalHolder} and hence completes the
  proof.
\end{proof}

\begin{rem}
  The hyperbolicity condition \eqref{eq:DiffeomorphismLocalization} we
  require is stronger than what we actually use. We could make do with
  the more complicated condition: there exists $\rho>1$ such that 
  \begin{enumerate}
  \item for all $x \in M$, all $y \in W^s(x)$, and all $n \geq 0$
    \begin{equation*}
      | \Delta_{x,y}^n |_{\rho^{-n}} \leq K \rho^{n}.
    \end{equation*}

  \item for all $x \in M$, all $y \in W^u(x)$, and all $ n \leq 0$ 
    \begin{equation*}
      | \Delta_{x,y}^n |_{\rho^{n}} \leq K \rho^{-n}.
    \end{equation*}
  \end{enumerate}
  This more complicated hyperbolicity condition is useful in the case
  of commutative groups endowed with matrix norms. 
\end{rem}

We now give a similar proof for Lie group valued cocycles over Anosov
flows. 

\begin{theorem}\label{thm:LivsicLieGroupFlow}
  Let $M$ be a compact Riemannian manifold, $f^t:M \rightarrow M$ be a
  $C^1$ topologically transitive $\lambda$-hyperbolic Anosov flow, and 
  $G$ be a Lie group. Let $\eta \in C^\alpha(M,\mathfrak{g})$.

  Define the cocycle $\Phi: M \times \R \rightarrow G$ by
  \begin{equation*}
    \frac{d}{dt} \Phi(x,t) = DR_{\Phi(x,t)}\eta(f^t x), \qquad \Phi(x,0)
    = \Id
  \end{equation*}
  and let $\Delta_{x,y}^t : \mathfrak{g} \rightarrow
  T_{\Phi^{-1}(x,t)\Phi(y,t)} G$ be given by
  \begin{equation*}
    \Delta_{x,y}^t = DL_{\Phi^{-1}(x,t)} DR_{\Phi(y,t)}. 
  \end{equation*}
  Assume the following localization condition; there exist $K, \rho
  >0$ such that for all $ x.y \in M$
  \begin{equation}
    \label{eq:FlowLocalization}
    \| \Delta_{x,y}^t \| \leq K e^{\rho |t|}
  \end{equation}
  where $\| \cdot \|$ is the standard operator norm. Suppose that for
  the pair $f^t$ and $\eta$:
  \begin{enumerate}
  \item The periodic orbit obstruction vanishes:
    \begin{quote}
      If $f^tp=p$ then $\Phi(p, t) = \Id$.
    \end{quote}

  \item The hyperbolicity condition is satisfied:
    \begin{equation}
      \label{eq:HypCondFlow}
      \rho - \lambda \alpha  < 0. 
    \end{equation}
  \end{enumerate}
  Then there exists $\phi \in C^\alpha(M,G)$ that solves
  \begin{equation}
    \label{eq:FlowCoboundary}
    \Phi(x,t) = \phi(f^tx) \phi^{-1}(x). 
  \end{equation}
\end{theorem}

\begin{proof}
  Let $x^* \in M$ be a point with a dense orbit $\mathcal{O}(x^*)$. If
  we fix $\phi(x^*)$ then, by \eqref{eq:FlowCoboundary}, we can define
  $\phi$ on $\mathcal{O}(x^*)$ by
  \begin{equation*}
    \phi(f^tx^*) = \Phi(x^*,t)\, \phi(x^*). 
  \end{equation*}
  Exactly as in the previous case, it suffices to show that there exist
  $\delta >0$ and $K >0$ such that 
  \begin{multline}
    \label{eq:FlowLocalHolder}
    \text{ if $d_M( f^{t+T}x^*, f^tx^*)<\delta$ then}\\
    d_G\bigl( \Phi(f^tx^*,T), \Id \bigr)  < K d_M( f^{t+T}x^*,
    f^tx^*)^\alpha. 
  \end{multline}
  Applying the Anosov Closing Lemma, Lemma \ref{AnosovClosingFlow},
  with $d_M( f^{t+T}x^*, f^tx^*)<\delta$ we obtain a periodic point $p
  \in M$ with $f^{T+\Delta}p=p$ and a point $z \in W^s(p) \cap
  W^u(f^tx^*)$. The periodic point satisfies:
  \begin{enumerate}
  \item $|\Delta| < K \, d_M( f^{t+T}x^*, f^tx^*)$.
    
  \item $d_M( f^{t+T} x^*, p) \leq K \, d_M( f^{t+T}x^*, f^tx^*)$.
    
  \end{enumerate}
  Let $C_F(s) := \Phi^{-1}(p,s) \cdot \Phi(z,s) $. Using the chain
  rule for functions of two variables we obtain
  \begin{align}\label{eq:CFdiffeq} 
    \frac{d}{ds}C_F(s) &= DL_{\Phi^{-1}(p,s)} DR_{\Phi(z,s)}\bigl[
    \eta(f^s z) - \eta(f^s p) \bigr]\\
    C_F(0)&=0 \nonumber
  \end{align}
  From the definition of $\Delta_{p,z}$ we obtain
  \begin{align*}
    \frac{d}{ds}C_F(s) &= \Delta_{p,z}^s \bigl[ \eta(f^s z) - \eta(f^s
    p) \bigr]
  \end{align*}
  and hence, for $s > 0$, we have the estimate
  \begin{align*}
    \label{eq:8}
    \Bigl\| \frac{d}{ds} C_F(s) \Bigr\| &\leq \| \Delta_{p,z}^s \|\,
    \|  \eta(f^s z) - \eta(f^s p) \| \\
    &\leq C_1 \, e^{\rho s} \| \, \eta(f^s z) - \eta(f^s p)
    \|&&\text{\eqref{eq:FlowLocalization}}\\
    &\leq C_2 \, e^{\rho s} \, d_M(f^s z, f^s p)^\alpha &&\eta \in
    C^\alpha(M,\mathfrak{g}) \\
    &\leq C_3 \,e^{(\rho - \lambda \alpha) s} \, d_M(z,p)^\alpha&&z \in W^s(p)\\
    &\leq C_4 \, e^{(\rho - \lambda \alpha) s} \,  d_M( f^{t+T}x^*,
    f^tx^*)^\alpha&&\text{Lemma \ref{AnosovClosingFlow}}
  \end{align*}
  As $C_F(0) = \Id$, using the hyperbolicity assumption
  \eqref{eq:HypCondFlow}, and the fact $d_G$ is a length metric we can
  integrate to get the following bound
  \begin{equation}
    \label{eq:9}
    d_G \bigl( C_F(s), \Id ) \leq \frac{C_4}{\lambda \alpha - \rho} d_M (
    f^{t+T}x^*, f^tx^*)^\alpha 
  \end{equation}
  and in particular
  \begin{align*}
    d_G\bigl( C_F(T), \Id)  &= d_G \bigl( \Phi^{-1}(p,T) \Phi(z,T),
    \Id)\\ 
    &\leq \frac{C_4}{\lambda \alpha - \rho} d_M ( f^{t+T}x^*,
    f^tx^*)^\alpha 
  \end{align*}
  As the periodic orbit obstruction is satisfied we have
  $\Phi(p,T+\Delta) = \Id$. By the cocycle property
  \begin{equation*}
  \Phi(p,T) = \Phi(p, -\Delta)
  \end{equation*}
  From the Anosov Closing Lemma, Lemma \ref{AnosovClosingFlow},
  we have $|\Delta|< K \, d_M ( f^{t+T}x^*, f^tx^*)$ and hence
  $\Delta$ is bounded. By compactness we can uniformly bound
  \begin{equation*}
    |L_{\Phi(p,T)}|_1=|L_{\Phi(p,-\Delta)}|_1 < C
  \end{equation*}
  Now we estimate 
  \begin{equation}
    \label{eq:FlowPhizEstimate}
    \begin{split}
      d_G \bigl( \Phi(z,T) , \Id \big) &\leq
      \begin{aligned}[t]
        &d_G \bigl( \Phi(p,T) \cdot \Phi^{-1}(p,T) \cdot \Phi(z,T), \Phi(p,T)
        \bigr)\\
        &\qquad + d_G \bigl( \Phi(p,T) , \Id \bigr)
      \end{aligned}
      \\
      & \leq |L_{\Phi(p,T)}|_1 \, d_G \bigl( C_F(T) , \Id
      \bigr) + d_G \bigl( \Phi(p,T)
      , \Id \bigr)\\
      & \leq C \,  d_M( f^{t+T}x^*, f^tx^*)^\alpha.
    \end{split}
  \end{equation}
  Since $ d_G \bigl( \Phi(z,T) , \Id \big)$ is bounded we can find a
  uniform estimate for $| R_{\Phi(z,T)} |_1$. 

  Let $C_R(s) := \Phi^{-1}(f^{t+T}x^*,-s) \Phi(f^Tz,-s)$.  We have
  \begin{align*}
   \frac{d}{ds} C_R(s) &= DL_{\Phi^{-1}(f^{t+T}x^*,-s)}
   DR_{\Phi(f^T z,-s)} \bigl[\eta(f^{t+T-s}x^*)-\eta(f^{T-s}z) \bigr] \\
   &= \Delta_{f^{t+T}x^*,f^T z}^{-s} \bigl[\eta(f^{t+T-s}x^*)-\eta(f^{T-s}z) \bigr]
  \end{align*}
  For $s>0$, we have
  \begin{align*}
    \Bigl\| \frac{d}{ds} C_R(s) \Bigr\| &\leq \|
    \Delta_{f^{t+T}x^*,f^T z}^{-s} \| \, \|
    \eta(f^{t+T-s}x^*)-\eta(f^{T-s}z)  \| \\
    &\leq C_1 \, e^{\rho s} \| \, \eta(f^{t+T-s}x^*)-\eta(f^{T-s}z)
    \|&& \text{\eqref{eq:FlowLocalization}}\\
    &\leq C_2 \, e^{\rho s} \,  d_M(f^{t+T-s}x^*,f^{T-s}z)^\alpha &&
    \eta \in C^\alpha(M,\mathfrak{g}) \\
    &\leq C_3 \, e^{(\rho - \lambda \alpha) s} \, d_M(f^{t+T}x^*,f^T
    z)^\alpha && f^Tz \in W^u (f^{t+T}x^*)\\
    &\leq C_4 \, e^{(\rho - \lambda \alpha) s} \, d_M( f^{t+T}x^*,
    f^tx^*)^\alpha &&\text{Lemma \ref{AnosovClosingFlow}}.
  \end{align*}
  By the hyperbolicity assumption we have $\rho - \lambda \alpha < 0$
  and hence we get the following bound, uniform in $s$
  \begin{equation*}
    d_G\bigl( C_R(s),\Id \bigr) \leq \frac{C_4}{\lambda \alpha - \rho}
    d_G( f^{t+T}x^*, f^tx^*)^\alpha. 
  \end{equation*}
  In particular, for $s=T$
  \begin{equation}\label{eq:FlowDTEstimate}
    d_G \bigl( C_R(T), \Id \bigr) \leq \frac{C_4}{\lambda \alpha - \rho}
    d_G( f^{t+T}x^*, f^tx^*)^\alpha. 
  \end{equation}
  Using the cocycle property we can rewrite $C_R(T)$ in the form 
  \begin{align*}
    C_R(T) &= \Phi^{-1}(f^{T+t}x^*,-T) \Phi(f^{T}z,-T)\\
    &= \Phi(f^t x^*,T) \Phi^{-1}(z,T)
  \end{align*}
  Finally using the triangle inequality, \eqref{eq:FlowPhizEstimate},
  and \eqref{eq:FlowDTEstimate} we get
  \begin{align*}
    d_G \bigl(\Phi(f^t x^*,T) , \Id \bigr) &\leq
    \begin{aligned}[t]
      &d_G \bigl( \Phi(f^t x^*,T) \Phi^{-1}(z,T) \Phi(z,T),\Phi(z,T)
      \bigr)\\ 
      & \qquad + d_G\bigl(\Phi(z,T), \Id \bigr)
    \end{aligned}
    \\
    & \leq | R_{\Phi(z,T)} |_1 \, d_G \bigl( C_R(T), \Id \bigr) +
    d_G\bigl(\Phi(z,T), \Id \bigr)\\
    & \leq C \, d_M( f^{t+T}x^*, f^tx^*)^\alpha.
  \end{align*}
  This completes the proof of Theorem \ref{thm:LivsicLieGroupFlow}. 
\end{proof}

One may deduce a version of the result for Anosov diffeomorphisms,
Theorem \ref{thm:LivsicLieGroupDiffeo}, from the one for
Anosov flows, Theorem \ref{thm:LivsicLieGroupFlow}, by a suspension
trick.  

\begin{rem}
  The hyperbolicity condition \eqref{eq:FlowLocalization} we
  require is stronger than what we actually use. We could make do with
  the more complicated condition: there exists $\rho>1$ such that 
  \begin{enumerate}
  \item for all $x \in M$, all $y \in W^s(x)$, and all $t \geq 0$
    \begin{equation*}
       \| \Delta_{x,y}^t \| \leq K e^{\rho t}
    \end{equation*}

  \item for all $x \in M$, all $y \in W^u(x)$, and all $t \leq 0$ 
    \begin{equation*}
       \| \Delta_{x,y}^t \| \leq K e^{-\rho t}
    \end{equation*}
  \end{enumerate}
  This more complicated hyperbolicity condition is useful in the case
  of commutative groups endowed with matrix norms. 
\end{rem}

\section{Verifying Localization}
\label{sec:localization}

The localization conditions \eqref{eq:DiffeomorphismLocalization} and
\eqref{eq:FlowLocalization} are formulated without any assumptions on
the metric. This has the advantage that the arguments apply equally
well to matrix norms, useful in computations in matrix Lie groups, and
to the left-invariant (or right-invariant) metrics so useful in geometric
computations. Finally these arguments also shed light on cases such as
diffeomorphism groups where the natural metric lacks the special
properties of either matrix norms or invariant metrics. 

In the case of matrix norms and invariant metrics we can easily relate 
the localization conditions \eqref{eq:DiffeomorphismLocalization} and
\eqref{eq:FlowLocalization} to properties of the generating map or
vector field. 

\subsection{Localization in Matrix Norms}
\label{sec:local-matr-norms}

\subsubsection{Diffeomorphism Case}

Let $G$ be a matrix Lie group endowed with a matrix norm. We will use
only the multiplicative property
\begin{equation*}
  \| A B \| \leq \|A \|\, \|B\|.
\end{equation*}
The operator $\Delta_{x,y}^n (g)$ satisfies 
\begin{align*}
  \| \Delta_{x,y}^n (g) \| &= \| \Phi^{-1}(x,n)\, g\, \Phi(y,n) \|\\
  &\leq \|\Phi^{-1}(x,n)\| \, \|  \Phi(y,n)\| \, \|g\|
\end{align*}
and hence we have 
\begin{align*}
  |\Delta_{x,y}^n| \leq  \|\Phi^{-1}(x,n)\| \, \|  \Phi(y,n)\|.
\end{align*}
If we let $\rho^2 = \max_{x \in M} \{ \|\eta(x)\|
,\|\eta^{-1}(x)\| \}$ then we have
\begin{equation*}
  \|\Phi(x,n)\| \leq \rho^{\frac{1}{2} |n|} \qquad \|\Phi^{-1}(x,n)\| \leq
  \rho^{\frac{1}{2} |n|} 
\end{equation*}
and hence 
\begin{equation*}
   |\Delta_{x,y}^n (g)| \leq \rho^{|n|}. 
\end{equation*}

\subsubsection{Flow Case}

Similarly, the operator $\Delta_{x,y}^t v$ satisfies 
\begin{align*}
  \| \Delta_{x,y}^t v \| &= \|  DL_{\Phi^{-1}(x,t)} DR_{\Phi(y,t)} v\|\\
  &\leq \|\Phi^{-1}(x,t)\| \, \|\Phi(y,t)\| \, \|v\|
\end{align*}
and hence we have 
\begin{align*}
  \|\Delta_{x,y}^t \| \leq  \|\Phi^{-1}(x,t)\| \, \|  \Phi(y,t)\|.
\end{align*}
If we let $2 \rho = \max_{x \in M} \{ \|\eta(x)\|
,\|\eta^{-1}(x)\| \}$ then we have
\begin{equation*}
  \|\Phi(x,t)\| \leq e^{ \frac{\rho}{2} |t|} \qquad \|\Phi^{-1}(x,t)\
  |\leq e^{ \frac{\rho}{2} |t|} 
\end{equation*}
and hence 
\begin{equation*}
   |\Delta_{x,y}^t v| \leq e^{ \rho |t|}
\end{equation*}

\subsubsection{Commutative Matrix Groups}

In the case of a commutative matrix group no localization assumption
is required. First observe that for commutative matrix groups we have 
\begin{equation*}
  |\Delta_{x,y}^n| \leq  \|\Phi^{-1}(x,n) \Phi(y,n)\|, \qquad
  \|\Delta_{x,y}^t \| \leq  \|\Phi^{-1}(x,t) \Phi(y,t)\|.
\end{equation*}
Thus the key is to estimate the quantity $\|\Phi^{-1}(x,n)
\Phi(y,n)\|$ in the case of an Anosov diffeomorphism, or the quantity
$ \|\Phi^{-1}(x,t) \Phi(y,t)\|$ in the case of an Anosov flow. 

In the case of a cocycle over an Anosov diffeomorphism we have the
evolution equation
\begin{equation*}
   \|\Phi^{-1}(x,n+1) \Phi(y,n+1)\| \leq \|\eta^{-1}(f^nx)
   \eta(f^ny)\| \,  \|\Phi^{-1}(x,n) \Phi(y,n)\|.
\end{equation*}
For $x \in M$, $y \in W^s(x)$, and $n \geq 0$ we have 
\begin{equation*}
  \|\eta^{-1}(f^nx)
   \eta(f^ny)\| \leq 1 + D \lambda^{\alpha n}. 
\end{equation*}
You can easily verify by induction that in this case for all $n \geq
0$ we have  
\begin{equation*}
  \|\Phi^{-1}(x,n) \Phi(y,n)\| \leq e^{D
    \frac{1-\lambda^(\alpha \,n)}{1-\lambda^\alpha}} 
\end{equation*}
and hence $|\Phi^{-1}(x,n) \Phi(y,n)\|< e^{\frac{D}{1-\lambda}}$. A
similar computation works when $x \in M$, $y \in W^u(x)$, and $n \leq
0$. 

In the case of flows we have the following evolution equation 
\begin{equation*}
  \frac{d}{dt} \|\Phi^{-1}(x,t) \Phi(y,t)\| \leq
  \|\eta(f^ty)-\eta(f^tx)\| \, \|\Phi^{-1}(x,t) \Phi(y,t)\|. 
\end{equation*}
For $x\in M$, $y \in W^s(x)$, and $t >0$ we have 
\begin{equation*}
  \|\eta(f^ty)-\eta(f^tx)\| < D \, e^{-\alpha \, \lambda \, t}. 
\end{equation*}
Check that in this case a version of the Gronwall inequality gives us 
\begin{equation*}
  \|\Phi^{-1}(x,t) \Phi(y,t)\| \leq e^{\frac{D}{\alpha \,\lambda} e^{
      - \alpha \, \lambda \, t}}
\end{equation*}
and hence $ \|\Phi^{-1}(x,t) \Phi(y,t)\| \leq e^{\frac{D}{\alpha
    \,\lambda}}$. A similar computation works when $x \in M$, $y \in
W^u(x)$, and $t \leq 0$.

\subsection{Localization in Right Invariant Norms}

A metric $d_G$on a topological group $G$ is called right invariant if
for all $f,g,h \in G$, $d_G(f \cdot h , g \cdot h) = d_g(f,g)$.
First observe that the local Lipshitz constant of the left
multiplication operator, and the operator norm of the differential map
of the left multiplication operator, are independent of the base point
since the metric is invariant under right multiplication and left and
right multiplication commute.

\subsubsection{Diffeomorphism Case}

If we let 
\begin{equation*}
\rho = \max_{x \in M} \max \{ |L_{\eta(x)}|_1,
|L_{\eta^{-1}(x)}|_1 \}
\end{equation*}
then by definition of $\Phi(x,n)$ we can write 
\begin{align*}
  |L_{\Phi(x,n)}|_{\rho^{-|n|}} &=
  \begin{cases}
    |L_{\eta(f^{n-1}x)} \circ \cdots \circ L_{\eta(x)}|_{\rho^{-|n|}}
    & n> 0 \\
    1& n=0\\
    |L_{\eta(f^nx)} \circ \cdots \circ L_{\eta(f^{-1}x)}|_{\rho^{-|n|}}
    & n < 0 \\
  \end{cases}\\
  &\leq\begin{cases} 
    |L_{\eta(f^{n-1}x)}|_1 \cdots |L_{\eta(x)}|_1 & n> 0 \\
    1& n=0\\
    |L_{\eta^{-1}(f^nx)}|_1 \cdots |L_{\eta^{-1}(f^{-1}x)}|_1 & n < 0 \\
  \end{cases}\\
  &\leq \rho^{|n|}.
\end{align*}
Since $R_g$ is an isometry the same estimate holds for our operator
$\Delta_{x,y}^n$ so
\begin{equation*}
  |\Delta_{x,y}^n|_{\rho^{-|n|}} \leq \rho^{|n|}.
\end{equation*}

\subsubsection{Flow Case}

Observe that since the metric is right invariant 
\begin{equation*}
   \| \Delta_{x,y}^t \| = \|D_{e}L_{\Phi^{-1}(x,t)}\|. 
\end{equation*}
If we let 
\begin{equation*}
  \rho := \max_{x \in M} \max_{t \in [-1,1]} \log \|D
  L_{\Phi^{-1}(x,t)} \|
\end{equation*}
then 
\begin{equation*}
  \| \Delta_{x,y}^t \| \leq e^{\rho \lceil |t| \rceil}  \leq
  e^{\rho} e^{\rho |t|}. 
\end{equation*}

\subsubsection{Commutative Group}

In a commutative group a right invariant metric is simply invariant
and hence the operators $\Delta_{x,y}^n$ and $\Delta_{x,y}^t$ are
isometries. Hence we can take $\rho=0$ and the hyperbolicity
conditions are automatically satisfied.

\section{Liv\v{s}ic Theory in Diffeomorphism Groups}
\label{sec:diffeomorphismgroups}

\subsection{Preliminaries on Diffeomorphism Groups}. 

We will recall some of the standard material on global analysis, see
for example \cite{Banyaga97}. 

We consider cocycles taking values in the group of $C^r$
diffeomorphisms of a compact Riemannian manifold $N$. 
The group operation is composition. As it is
well known, the group operation is continuous 
but not differentiable \cite{delaLlaveO99}. 
Hence the previous results do not apply directly. 
Nevertheless, we will see that the rough lines of 
the technique can be applied, but we get some lower
regularity of the solutions. 

The group of $C^r$ diffeomorphisms of a compact Riemannian manifold
$N$ has the structure of a Banach manifold modeled on the space
$T_hC^r(N,N)$, defined by
\begin{equation*}
  T_hC^r(N,N) = \{ v \in C^r(N,TN) : \pi_N \circ v = h \}.
\end{equation*}

We will endow with $\diff^r(N)$ with a length metric induced from the
Riemannian structure on $N$. Given $h \in \diff^r(N)$ and $y \in N$
there exists a neighborhood $U \subset T_y N$ sufficiently small that
a local representative $\tilde h_y : U \rightarrow T_{h(y)}N$ is
uniquely defined by
\begin{equation*}
  h \bigl( \exp_y v \bigr) = \exp_{h(y)} \bigl(\tilde h_y (v)\bigr).
\end{equation*}
Since $\tilde h_y$ is defined between Banach spaces we can
differentiate it in the usual manner. We will always use $D$ to denote
differentiation in the manifold $N$.

We thus obtain
\begin{equation*}
  D^nh(y) := D^n\tilde h(0) : (T_yN)^{\otimes n} \rightarrow T_{h(y)} N.
\end{equation*}
The derivative produced in this fashion coincides with the usual
notion of covariant derivative defined by the Levi-Civita
connection. 

When dealing with a smooth curve $h: \R \rightarrow \diff^r(N)$ we
modify this idea slightly. For any $s \in \R$ and any $y \in N$ there
exists a neighborhood $V$ of $s$ and $0 \in U \subset T_y N$ such
that for any $t \in V$ the local representative $\tilde h(t)_y : U
\subset T_yN \rightarrow T_{h(s)(y)}N$ is defined uniquely by
\begin{equation*}
  h(t) \bigl( \exp_y v \bigr) = \exp_{h(s)(y)} \bigl(\tilde h(t)_y (v)\bigr).
\end{equation*}
We may therefore differentiate with respect to $t$ to obtain 
\begin{equation*}
  \frac{d}{dt} \tilde h(t)_y \Bigr|_{t=s}:  U \subset T_yN \rightarrow
  T_{h(s)(y)}N.
\end{equation*}
We declare 
\begin{equation*}
  \frac{d}{ds} D^n_y h(s) := D^n \frac{d}{dt} \tilde
  h(t)_y  \Bigr|_{t=s} (0).
\end{equation*}

\subsection{Metric on $\diff^r(N)$}

Let $p : [0,1] \rightarrow \diff^r(N)$ be a piecewise $C^1$ path. This
is equivalent to $ \frac{d}{ds} D^n p_s $ piecewise continuous in $s$
for $0\leq n \leq r$.  We can define the length of such a piecewise
$C^1$ path by
\begin{equation*}
  \ell_r(p) = \max_{0\leq n \leq r} \max_{y \in N} \int_0^1
  \|\frac{d}{ds}D^n_y p_s  \| \, ds  
\end{equation*}
where the norm is the appropriate operator norm induced by the
Riemannian metric. If we compute the length of only a part of the path
then we write
\begin{equation*}
  \ell_r(p;s) = \max_{0\leq n \leq r} \max_{y \in N} \int_0^s
  \|\frac{d}{dt}D^n_y p_t \|\, dt  
\end{equation*}

When we wish to compute $\frac{d}{dt} \|D_y^k p_t \|$ we can use the
local representative but must take care to consider the lift of the
appropriate norm. Let $p \in N$ be an arbitrary point and $q$ close
enough to $p$ that we may consider the lift of $q$ to the neighborhood
$U \subset T_p N$ on which the local representative is defined. Since
the Riemannian metric is smooth we can find a globally defined
constant $\kappa>0$, depending on the Riemannian metric, such that
\begin{equation}
  \label{eq:GeometricEstimate}
  \| \cdot \|_{q} \leq \| \cdot \|_p + \kappa \, \| \cdot \|_p \, d_N(q,p). 
\end{equation}
In fact we will only need the infinitesimal version of this.

Notice
\begin{equation*}
   \ell_0(p) = \max_{y \in N} \int_0^1 \|\frac{d}{ds}p_s(y)\|_{p_s(y)} \, ds
\end{equation*}
is precisely the maximum over all $y \in N$ of the usual length of the
path $p_s(y)$ in $N$. We use this length structure on $\diff^r(N)$ to
induce a metric by defining
\begin{equation}
  \label{eq:drDefn}
  d_r(g,h) := \inf_{p \in \mathfrak{P}}  \max \{\ell(p),\ell(p^{-1}) \}
\end{equation}
where 
\begin{equation*}
  \mathfrak{P} = \bigl\{ p \in C^1_\mathrm{pw} \bigl( [0,1] , \diff^r(N)
  \bigl), p_0=g, p_1=h \bigr\}.
\end{equation*}
Notice that our definition has the symmetry property 
\begin{equation*}
  d_r(g,h) =  d_r(g^{-1},h^{-1}).
\end{equation*}
It is worth noticing that for paths which connect a diffeomorphism $f$
to the identity $\ell_0(p)=\ell_0(p^{-1})$.  Furthermore, if $f$ and $g$
are sufficiently $C^1$ close there is a standard way of producing an
interpolating path, namely
\begin{equation*}
  p_s(y) := \exp_{f(y)}\bigl( s \exp_{f(y)}^{-1} g(y) \bigr). 
\end{equation*}
Since geodesics are locally distance minimizing for $f$ and $g$
sufficiently close this path is the path along which $d_0$ is
minimized. In this case, we have 
\begin{equation*}
  d_0(f,g) = \max \bigl\{ \max_{y \in N} d_N\bigl(f(y),g(y)\bigr), \max_{y
  \in N} d_N\bigl(f^{-1}(y),g^{-1}(y)\bigr) \bigr\}.
\end{equation*}
If $f$ and $g$ are not sufficiently close then this may no longer
necessarily true.

Our first lemma gives explicit estimates on the size of derivatives in
an interpolating path. 

\begin{lemma}\label{lem:PathBounds}
  For all piecewise $C^1$ paths $p : \R
  \rightarrow \diff^r(N)$, all $1 \leq k \leq r$, and all $s >0 $
  \begin{equation*}
    \|D^k p_s \| \leq e^{\kappa \,  \ell_0(p;s)}  \, \bigl(
    \ell_k(p;s) + \|D^k p_0\| \bigr) 
  \end{equation*}
  where the norms are the operator norms for $D^k p_s, D^k p_0 :
  (TN)^{\otimes k} \rightarrow TN$, and $\kappa$ is the constant from
  \eqref{eq:GeometricEstimate}. In particular
  \begin{equation*}
    \|D p_s \|_{r-1} \leq e^{\kappa \,  \ell_0(p;s)}  \, \bigl(
    \ell_r(p;s) + \|D p_0\|_{r-1} \bigr) 
  \end{equation*}
\end{lemma}

\begin{proof}
  Let $0 \leq k \leq r$. Let $y \in N$ and $v \in (T_yN)^{\otimes k}$
  be arbitrary. The estimate \eqref{eq:GeometricEstimate} gives us
   \begin{multline*}
     \| D^k_y p_{t+\epsilon} v \|_{p_{t + \epsilon}(y)}\\
     \leq \| D^k_y p_{t+\epsilon} v \|_{p_t(y)} + \kappa \, \| D^k_y
     p_{t+\epsilon} v \|_{p_t(y)} \, d_N \bigl( p_{t+ \epsilon}(y),
     p_t(y) \bigr)
   \end{multline*}
   and hence 
   \begin{align*}
     \| D^k_y p_{t+\epsilon} v \|_{p_{t + \epsilon}(y)} - \|D^k_y
     p_t v\|_{p_t(y)} 
     &\leq \| D^k_y p_{t+\epsilon} v \|_{p_t(y)} - \|D^k_y p_t v\|_{p_t(y)}\\
     &+ \kappa \, \| D^k_y p_{t+\epsilon} v \|_{p_t(y)} \, d_N \bigl( p_{t+
       \epsilon}(y), p_t(y) \bigr).
   \end{align*}
   Since the terms on the right hand side have the same base point
   the triangle inequality applies 
   \begin{equation*}
      \| D^k_y p_{t+\epsilon} v \|_{p_t(y)} - \|D^k_y p_t v\|_{p_t(y)}
      \leq  \| D^k_y p_{t+\epsilon} v -D^k_y p_t v\|_{p_t(y)}
   \end{equation*}

   Dividing by $\epsilon$ and taking the limit we obtain 
   \begin{equation*}
     \frac{d}{dt} \|D_y^k p_t v\|_{p_t(y)} \leq \| \frac{d}{dt} D^k_y
     p_t v \|_{p_t(y)} + \kappa \, \|D_y^k p_t v\|_{p_t(y)} \|\frac{d}{dt}
     p_t(y) \|_{p_t(y)}.   
   \end{equation*}
   The classical Gronwall inequality therefore gives us
   \begin{align*}
     \|D_y^k p_s v\|_{p_s(y)} &\leq e^{ \kappa \, \int_0^s \|\frac{d}{dt}
       p_t(y) \|_{p_t(y)}\, dt } \, \bigl( \int_0^s \| \frac{d}{dt}
     D^k_y
     p_t v \|_{p_t(y)} \, dt + \|D^k_y p_0 v \|_{p_0(y)} \bigr) \\
     &\leq e^{\kappa \, \ell_0(p;s)} \, \bigl( \ell_k(p;s) + \|D^k_y p_0 v
     \|_{p_0(y)} \bigr)
   \end{align*}
   Finally we take a supremum over all $v \in (T_yN)^{\otimes k}$ with
   $\|v\|_y = 1$ and then a supremum over $y \in N$.
\end{proof}

Our second lemma contains the central part of a version of the mean
value theorem for our metric. 

\begin{lemma}\label{lem:TechnicalMVT}
 Let $p \in C^1\bigl( [0,1], \diff^{r-1}(N) \bigr)$ and let $h \in
 \diff^r(N)$. Then, 
 \begin{equation*}
   \ell_{r-1}(h \circ p_s) \leq C\, \|D h\|_{r-1} \, (1+\max_{s \in [0,1]}\|D
   p_s\|_{r-2})^{r-1}\, \ell_{r-1}(p_s). 
 \end{equation*}
 Let $p \in C^1 \bigl( [0,1], \diff^{r}(N) \bigr)$ and let $h \in
 \diff^r(N)$. Then,
 \begin{equation*}
   \ell_r(p_s \circ h) \leq C \, \max_{k_1, \cdots, k_r} \|D^1
   h\|^{k_1} \cdots \|D^r h\|^{k_r} \, \ell_r(p_s)
 \end{equation*}
 where the max is taken over all $k_1, \dots, k_r \geq 0$ such that
 \begin{equation*}
   k_1 + 2 k_2 + \cdots + r k_r \leq r.
 \end{equation*}
 Crudely, this may be estimated by
 \begin{equation*}
   \ell_r(p_s \circ h) \leq C \, ( 1+ \|D h\|_{r-1})^r\, \ell_r(p_s) .
 \end{equation*}
 In each case, the constant $C$ depends on $r$. 
\end{lemma}

\begin{proof}
  To determine $\ell_r( p_s \circ h)$ we need to compute
  $\frac{d}{ds}D^n[ p_s \circ h]$ for $0 \leq n \leq r$. We apply the
  Fa\`a di Bruno formula to $p_s \circ h$ to obtain
 \begin{equation*}
   D^n [p_s \circ h] = \sum_{k_1, \dots, k_n} D^k p_s \circ h \cdot [
   D^1h^{\otimes k_1} \otimes \cdots \otimes D^nh^{\otimes k_n}] 
 \end{equation*}
 where $ k= k_1 + \cdots + k_n$ and the sum is taken over all $k_1,
 \dots, k_n$ such that $k_1 + 2 k_2 + \cdots + n k_n =
 n$. We use $\circ$ to denote composition in the base space $N$ and
 $\cdot$ to indicate composition (multiplication) in the space of
 linear operators. Differentiating with respect to $s$, we obtain
 \begin{equation*}
   \frac{d}{ds} D^n [p_s \circ h] = \sum_{k_1, \dots, k_n} C_{k_1,
     \dots, k_n}\,\frac{d}{ds}
   D^k p_s \circ h \cdot \bigl[ D^1h^{\otimes k_1} \otimes \cdots \otimes
   D^nh^{\otimes k_n} \bigr]. 
 \end{equation*}
 We have the following estimate
 \begin{align*}
   \bigl\| D^1h^{\otimes k_1} \otimes \cdots \otimes
   D^nh^{\otimes k_n} \bigr\| &= \|D^1h\|_0^{k_1}\,\cdots\,
   \|D^nh\|_0^{k_n}\\
   &\leq (1+\|Dh\|_{n-1})^n\\
   &\leq ( 1 + \|Dh\|_{r-1})^r .
 \end{align*}
 Since 
 \begin{equation*}
   \int_0^1 \| \frac{d}{ds}
   D^k p_s \circ h \| \, ds \leq \ell_k( p_s) \leq \ell_n (p_s) \leq \ell_r(p_s)
 \end{equation*}
 we combine our estimates to obtain
 \begin{equation*}
  \ell_r(p_s \circ h) \leq C \, (1+ \|D h\|_{r-1})^r \ell_r (p_s).
 \end{equation*}
 The constant $C$ depends only on $r$. 

 To determine $\ell_{r-1}( h \circ p_s)$ we need to compute
  $\frac{d}{ds}D^n[ h \circ p_s]$ for $0 \leq n \leq r-1$. We apply the
  Fa\`a di Bruno formula to $h \circ p_s$ to obtain
 \begin{equation*}
   D^n [h \circ p_s] = \sum_{k_1, \dots, k_n}  C_{k_1,
     \dots, k_n}\, D^k h \circ p_s \cdot [
   D^1p_s^{\otimes k_1} \otimes \cdots \otimes D^np_s^{\otimes k_n}] 
 \end{equation*}
 where $ k= k_1 + \cdots + k_n$ and the sum is taken over all $k_1,
 \dots, k_n$ such that $k_1 + 2 k_2 + \cdots + n k_n =
 n$. Differentiating with respect to $s$, we obtain
 \begin{multline*}
  \frac{d}{ds} D^n [h \circ p_s] = \sum_{k_1, \dots, k_n} \bigl[
  D^{k+1} h \circ p_s \cdot [ \frac{d}{ds}p_s \otimes
   D^1p_s^{\otimes k_1} \otimes \cdots \otimes D^np_s^{\otimes k_n}] +\\
   D^k h \circ p_s \cdot \frac{d}{ds}[ D^1p_s^{\otimes k_1} \otimes
   \cdots \otimes D^np_s^{\otimes k_n}] \bigr]
 \end{multline*}
 The term
 \begin{equation*}
   \frac{d}{ds}[ D^1p_s^{\otimes k_1} \otimes \cdots \otimes
   D^np_s^{\otimes k_n}]
 \end{equation*}
 consists of $k$ terms, each of which has a single term of the form
 $\frac{d}{ds} D^l p_s$ and thus can be estimated by 
 \begin{equation*}
   \|D^1p_s\|^{k_1} \cdots \|D^l p_s \|^{k_l -1} \cdots
   \|D^np_s\|^{k_n} \, \bigl\| \frac{d}{ds} D^lp_s \bigr\|.
 \end{equation*}
 As above we can estimate 
 \begin{multline*}
   \|D^1p_s\|^{k_1} \cdots \|D^l p_s \|^{k_l -1} \cdots
   \|D^np_s\|^{k_n}\\
   \leq (1 + \max_{s \in [0,1]} \| Dp_s\|_{n-1} )^n \leq (1 + \max_{s
     \in [0,1]} \| Dp_s\|_{r-2} )^{r-1}
 \end{multline*}
 and 
 \begin{equation*}
   \int_0^1 \bigl\| \frac{d}{ds} D^lp_s \bigr\| \, ds \leq \ell_{r-1}(p_s)
 \end{equation*}
 for $0 \leq l \leq r-1$. Finally 
 \begin{equation*}
   \| D^{k+1} h \circ p_s \| ,  \|D^{k} h \circ p_s\| \leq \|D
   h\|_{r}.  
 \end{equation*}
 Combining these estimates we get the required result.
\end{proof}

Finally we can combine these two to give a more convenient form for
the mean value theorem. 

\begin{lemma}\label{lem:MVT}
  Let $C>0$ and $r \in \N$ be arbitrary. Suppose $h \in \diff^r(N)$ and
  $g_1, g_2 \in \diff^{r-1}(N)$. There exists a constant $C'>0$ such
  that if 
  \begin{equation*}
    d_r(h, \Id) < C,\quad d_{r-1}(g_1,\Id) <C,  \quad d_{r-1}(g_2,\Id)<C
  \end{equation*}
  then   
  \begin{align*}
    d_{r-1}(h \circ g_1, h \circ g_2) &< C' \, d_{r-1} (g_1,g_2),\\
    d_{r-1}(g_1 \circ h, g_2 \circ h) &< C' \, d_{r-1} (g_1,g_2).
  \end{align*}
  The constant $C'$ depends on $C$, $r$, and the manifold $N$. 
\end{lemma}

\begin{proof}
  Since $d_{r-1} (g_1,g_2) < 2 C$ we may take a path $p_s$ joining
  $g_1$ to $g_2$ with $\ell_{r-1}(p_s) < 2 C$ and $
  \ell_{r-1}(p_s^{-1})<2 C$. Using Lemma \ref{lem:PathBounds} we see
  that there exists $C_1>0$ with $\|D p_s\|_{r-2} < C_1$ and $\| D
  p_s^{-1} \|_{r-2} < C_1 $. Again by Lemma \ref{lem:PathBounds} since
  $d_{r}(h, \Id) < C$ there exists $C_2 > 0$ such that $\|Dh\|_{r-1}<
  C_2$. Now by Lemma \ref{lem:TechnicalMVT} there exists $C_3>0$,
  depending only on $r$, such that
  \begin{align*}
    \ell_{r-1} (h \circ p_s) &\leq C_3 \, \|Dh\|_{r-1} \, (1 + \|D
    p_s\|_{r-2})^{r-1} \, \ell_{r-1}(p_s), \\
    \ell_{r-1} (p_s^{-1} \circ h^{-1}) &\leq C_3 \, \bigl(1+
    \|Dh^{-1}\|_{r-2}\bigr)^{r-1} \, \ell_{r-1}(p_s^{-1}). 
  \end{align*}
  All these quantities are bounded so we have a $C'>0$ such that 
  \begin{align*}
    \ell_{r-1} (h \circ p_s) &\leq C' \, \ell_{r-1}(p_s), \\
    \ell_{r-1} (p_s^{-1} \circ h^{-1}) &\leq C' \, \ell_{r-1}(p_s^{-1}).
  \end{align*}
  Taking infimums we obtain 
  \begin{equation*}
    d_{r-1}(h \circ g_1, h \circ g_2)  \leq C' d_{r-1}(g_1,g_2). 
  \end{equation*}
  The other direction is an immediate consequence of the symmetry of
  our metric.
\end{proof}

\subsection{Preliminary Estimates for the Flow Case}

\begin{lemma}\label{lem:DiffeomorphismGroup}
  Let $\eta \in C^\alpha(M, \mathfrak{X}^r(N))$ and define
  \begin{align*}
    \label{eq: DiffeomorphismHyperbolicityCondition}
    \rho_0 &= \max_{x \in M} \max_{y \in N} \|\eta_{x} (y) \|\\
    \rho_1 &= \max_{x \in M} \max_{y \in N} \|D_y \eta_{x}\|.
  \end{align*}
  Define the cocycle $\Phi : M \times \R \rightarrow \diff^r(N)$ by
  \begin{equation*}
    \frac{d}{ds} \Phi(x,s) = \eta_{f^s x} \circ \Phi(x,s),\qquad
    \Phi(x,0) = \Id.  
  \end{equation*}
  Then we have the following estimates for $n \leq r$
  \begin{align*}
    d_0 ( \Phi(x,s) , \Id ) \leq \rho_0 \, |s|, \qquad
    \| D^n \Phi (x,s)\| \leq C \, e^{ n (\rho_1 + \kappa  \, \rho_0) |s|}
  \end{align*}
  where $C$ is a constant that depends only on $r$ and $\kappa$ is the
  geometric constant introduced above. 
\end{lemma}

\begin{proof}
  To aid in readability we will write $\Phi_s$ for $\Phi(x,s)$ since
  $x$ plays no r\^ole in this lemma.  We immediately have
  \begin{equation}\label{eq:VelocityBound}
    \Bigl\|\frac{d}{ds} \Phi_s(y)\Bigr\|_{\Phi_s(y)} \leq \rho_0
  \end{equation}
  which, upon integrating, establishes the first estimate. 

  We proceed by induction to establish the remaining estimates. For
  fixed $y \in N$ and $v \in T_y N$ we have 
  \begin{equation*}
    \frac{d}{ds} D_y \Phi_s(y)\, v = D_{\Phi_s(y)} \eta_{f^s x} \cdot
    D_y \Phi_s \, v , \qquad D \Phi(x,0) = \Id 
  \end{equation*}
  and thus 
  \begin{equation}
    \begin{aligned}
      \|\frac{d}{ds} D_y \Phi_s \, v \|_{\Phi_s(y)} = \|D_{\Phi_s}
      \eta_{f^s x}\| \, \|D_y \Phi_s \, v\|_{\Phi_s(y)} \leq \rho_1 \,
      \|D_y \Phi_s \, v\|_{\Phi_s(y)}
    \end{aligned}\label{eq:NormDerivativeBounds}
  \end{equation}

  Exactly as in Lemma \ref{lem:PathBounds}, using
  \eqref{eq:GeometricEstimate} yields the following estimate,
  \begin{multline*}
    \frac{d}{ds} \|D_y\Phi_s \, v\|_{\Phi_s(y)} \\
    \leq \|\frac{d}{ds} D_y\Phi_s \, v \|_{\Phi_s(y)} + \kappa \, \|
    D_y \Phi_s\, v \|_{\Phi_s(y)} \, \|\frac{d}{ds} \Phi_s(y) \|_{
      \Phi_s(y)}.
  \end{multline*}
  Using \eqref{eq:VelocityBound} and \eqref{eq:NormDerivativeBounds} we get 
  \begin{equation*}
     \frac{d}{ds} \|D_y\Phi_s \, v\|_{\Phi_s(y)} \leq
    \rho_1  \|D_y\Phi_s \, v\|_{\Phi_s(y)} 
    + \kappa \, \rho_0 \, \| D_y \Phi_s\, v \|_{\Phi_s(y)} 
  \end{equation*}
  The Gronwall inequality gives 
  \begin{equation*}
    \| D_y \Phi_s\, v \| \leq e^{(\rho_1 + \kappa \, \rho_0) |s|} 
  \end{equation*}
  which establishes the base case. 

  Applying the Fa\`a di Bruno formula to $\eta_{f^sx}\circ \Phi_s$ we
  obtain
  \begin{equation*}
    \frac{d}{ds} D^n_y [\eta_{f^sx}\circ \Phi_s] = \sum_{k_1, \dots,
      k_n}  C_{k_1, \dots, k_n} \, D^k_{\Phi_s(y)} \eta_{f^s x} \cdot
    \bigl( D_y \Phi_s  
    \bigr)^{\otimes k_1} \otimes \cdots \otimes \bigl(D^n_y \Phi_s
    \bigr)^{\otimes k_n} 
  \end{equation*}
  where $k=k_1 + \cdots + k_n$ and the sum is taken over all $k_1,
  \dots, k_n$ such that $k_1 + 2 k_2 + \cdots + n k_n = n$.

  Thus we obtain for any 
  \begin{align*}
    \| \frac{d}{ds} D^n_y \Phi_s \| & \leq \sum_{k_1, \dots, k_n}
    C_{k_1, \dots, k_n} \, \| D^k_{\Phi_s(y)} \eta_{f^s x} \| \, \|
    D_y \Phi_s\|^{k_1} \cdots \|D^n_y \Phi_s\|^{k_n}\\
    \intertext{Separating this into $k_n=1$ and $k_n$=0 terms we
      obtain} \| \frac{d}{ds} D^n_y \Phi_s \|&
     \begin{aligned}[t]
       \leq &\|D_{\Phi_s(y)}\eta_{f^sx}\| \, \|D^n_y \Phi_s\| \\
       & + \sum_{k_n=0} C_{k_1, \dots, k_{n-1}} \, \| D^k_{\Phi_s(y)}
       \eta_{f^s x} \| \, \| D_y \Phi_s\|^{k_1} \cdots \|D^{n-1}_y
       \Phi_s\|^{k_{n-1}}
     \end{aligned}\\
    & \leq \rho_1 \, \|D^n_y \Phi_s\| + C \,  e^{ n (\rho_1 + \kappa \, \rho_0) |s|}.
  \end{align*}
  Again, as in Lemma \ref{lem:PathBounds}, we obtain for any $v \in
  (T_yN)^{\otimes n}$ 
  \begin{align*}
    \frac{d}{ds} \| D^n_y \Phi_s \, v\|_{\Phi_s(y)} &\leq \|
    \frac{d}{ds} D^n_y \Phi_s\, v \|_{\Phi_s(y)} + \kappa \, \|D^n_y
    \Phi_s\, v \|_{\Phi_s(y)} \, \|\frac{d}{ds} \Phi_s(y)
    \|_{\Phi_s(y)}\\
    & \leq \rho_1 \, \|D^n_y \Phi_s \, v\| + C \, e^{ n (\rho_1 +
      \kappa \, \rho_0) |s|} + \kappa \, \rho_0 \, \|D^n_y \Phi_s \,
    v\|
  \end{align*}

  Using a Gronwall-type inequality we then obtain 
  \begin{align*}
    \| D^n_y \Phi_s \| \leq C \, e^{n (\rho_1 + \kappa \, \rho_0) |s|}
  \end{align*}
  as required.
\end{proof}

\subsection{Preliminary Estimates for the Diffeomorphism Case}

\begin{lemma}\label{lem:PreliminaryDiffeo}
  Let $\eta \in C^\alpha(M, \diff^r(N)$ and define
  \begin{align*}
    \rho_0 &= \max_{x \in M} d_0(\eta_x, \Id )\\
    \rho_1 &= \max_{x \in M} \max \{ \|D \eta_x\|, \|D \eta_x^{-1}\| \}. 
  \end{align*}
  Define the cocycle $\Phi: M \times \Z \rightarrow \diff^r(N)$ by 
  \begin{equation*}
    \Phi(x,n) =
    \begin{cases}
      \eta_{f^{n-1}x} \circ \cdots \circ \eta_x & n \geq 1\\
      \Id & n=0\\
      \eta_{f^nx}^{-1} \circ \cdots \circ \eta_{f^{-1}x}^{-1} & n \leq -1
    \end{cases}.
  \end{equation*}
  The we have the following estimates for $m \leq r$
  \begin{equation*}
    d_0(\Phi(x,n), \Id) \leq \rho_0 |n| \qquad \|D^m \Phi(x,n) \|\leq
    C \rho_1^{m \, |n|}. 
  \end{equation*}
\end{lemma}

\begin{proof}
  By the triangle inequality we have for $n \geq 1$
  \begin{align*}
    d_0(\Phi(x,n), \Id) &\leq d_0\bigl(\Phi(x,n),\Phi(x,n-1)\bigr) +
    \cdots + d_0\bigl(\Phi(x,1), \Id \bigr)\\
    &\leq d_0\bigl(\eta_{f^{n-1}x} \circ \Phi(x,n-1), \Phi(x,n-1)
    \bigr) + \cdots + d_0\bigl(\eta_x, \Id \bigr)\\
    &\leq d_0\bigl(\eta_{f^{n-1}x}, \Id \bigr) + \cdots +
    d_0\bigl(\eta_x, \Id \bigr)\\
    &\leq n \max_{x \in M} d_0(\eta_x,\Id).
  \end{align*}
  Since $d_0(\eta_x, \Id) = d_0(\eta_x^{-1},\Id)$ a similar argument
  works for $n \leq 0$ too. 
  
  We have the following basic evolution equation
  \begin{equation*}
    \Phi(x, n+1) = \eta_{f^nx} \circ \Phi(x,n)
  \end{equation*}
  and hence we have
  \begin{equation*}
    D \Phi(x,n+1) = D\eta_{f^nx} \circ \Phi(x,n) \cdot D \Phi(x,n). 
  \end{equation*}
  Taking norms, we get 
  \begin{align*}
   \| D \Phi(x,n+1) \| &\leq \|D\eta_{f^nx}\| \, \|D \Phi(x,n) \|\\
   & \leq \rho_1 \, \|D \Phi(x,n) \|.
  \end{align*}
  Finally since $\Phi(x,0) = \Id$, and hence $\|D \Phi(x,0) \|=1$, we
  see that we have
  \begin{equation*}
    \|D \Phi(x,n) \| \leq \rho_1^n    
  \end{equation*}
  for all $n \geq 0$. Observe that for $n \leq 0$ we have the
  following evolution equation
  \begin{equation*}
    \Phi(x,n-1) = \eta_{f^{n-1} x}^{-1} \circ \Phi(x, n)
  \end{equation*}
  so exactly as above we obtain 
  \begin{align*}
    \|D \Phi(x,n-1) \| &\leq \|D \eta_{f^{n-1}x}^{-1}\| \, \|D
    \Phi(x,n)\|\\
    &\leq \rho_1 \, \|D \Phi(x,n)\|.
  \end{align*}
  Hence we get
  \begin{equation*}
    \|D \Phi(x,n) \| \leq \rho_1^{|n|}    
  \end{equation*} 
  for $n \leq 0$.
  
  We now proceed by induction. Suppose that we have the estimate 
  \begin{equation*}
    \|D^k \Phi(x,n) \| \leq C \, \rho_1^{k \, n} 
  \end{equation*}
  for all $k < m$. We will now establish the estimate for $m$. 
  Applying the Fa\`a di Bruno formula to our basic evolution equation
  we obtain
  \begin{align*}
    D^m \Phi(x, n+1) = \sum_{k_1, \cdots, k_m} &D^k\eta_{f^nx}\circ
    \Phi(x,n)\\
    &\cdot \Bigl[\bigl(D^1 \Phi(x,n)\bigr)^{\otimes k_1} \otimes
    \cdots \otimes \bigl(D^m \Phi(x,n) \bigr)^{\otimes k_m} \bigr]
  \end{align*}
  where $k = k_1 + \cdots k_m$ and the sum is taken over all $k_1,
  \dots, k_m \geq 0$ such that $k_1 + 2 k_2 + \dots + m k_m =
  m$. Either $k_m=0$ or $k_m=1$. We separate these terms 
  \begin{multline*}
    D^m \Phi(x, n+1) =\\ D^1 \eta_{f^n x} \circ \Phi(x,n) \cdot
    D^m\Phi(x,n) + \sum_{k_1, \cdots, k_{m-1}} D^k\eta_{f^nx}\circ
    \Phi(x,n) \\
    \cdot \Bigl[\bigl(D^1 \Phi(x,n)\bigr)^{\otimes k_1} \otimes \cdots
    \otimes \bigl(D^{m-1} \Phi(x,n) \bigr)^{\otimes k_{m-1}} \bigr].
  \end{multline*}
  Taking norms we obtain 
  \begin{multline*}
    \|D^m \Phi(x, n+1)\| \leq \|D^1 \eta_{f^n x}\| \, \|
    D^m\Phi(x,n)\| \\+  \sum_{k_1, \cdots, k_{m-1}}  \| D^k\eta_{f^nx}
    \| \|D^1 \Phi(x,n) \|^{k_1} \cdots \|D^{m-1} \Phi(x,n)\|^{k_{m-1}} 
  \end{multline*}
  Now applying the inductive assumption we obtain
  \begin{equation*}
    \|D^m \Phi(x, n+1)\| \leq \rho_1 \, \| D^m\Phi(x,n)\| + C \,
    \rho_1^{m n}
  \end{equation*}
  where $C$ depends on $m$ and on $\eta$.  Using this we can check that if 
  $$\| D^m\Phi(x,n)\| \leq  \frac{C}{\rho_1^m - \rho_1} \rho_1^{m \, n}$$
  then 
  $$\| D^m\Phi(x,n+1)\| \leq \frac{C}{\rho_1^m - \rho_1} \rho_1^{m \,
    (n+1)}.$$
  Since the estimate is obviously true for $n=0$ we have established
  it for all $n \geq 0$. Starting with the evolution equation
  \begin{equation*}
    \Phi(x,n-1) = \eta_{f^{n-1}x}^{-1} \circ \Phi(x,n) 
  \end{equation*}
  and applying the same estimates establishes the result for all $n
  \leq 0$. 

\end{proof}

\subsection{Main Theorem for Diffeomorphism Valued Cocycles}

Let $\mathfrak{X}^r(N)$ denote the space of $C^r$ vector fields on
$N$.

\begin{theorem}\label{thm:DiffeomorphismLivsic}
  Let $M$ be a compact Riemannian manifold with $f^t:M \rightarrow M$
  be a $C^1$ topologically transitive $\lambda$-hyperbolic Anosov
  flow. Let $N$ be a compact Riemannian manifold.

  Given a $\eta \in C^\alpha \bigl( M,\mathfrak{X}^r(N) \bigr)$ define
  a cocycle $\Phi: M \times \R \rightarrow \diff^r(N)$ by
  \begin{equation*}
    \frac{d}{dt} \Phi(x,t) = \eta(f^t x) \circ \Phi(x,t), \qquad \Phi(x,0)
    = \Id.
  \end{equation*}
  
  Let
  \begin{align*}
    \rho_0 &= \max_{x \in M} \max_{y \in N} \| \eta_x(y) \|\\
    \rho_1 &= \max_{x \in M} \max_{y \in N} \| D_y \eta_x \|.
  \end{align*}

  Suppose that for the pair $f^t$ and $\eta$:
  \begin{enumerate}
  \item The periodic orbit obstruction vanishes:
    \begin{quote}
      If $f^tp=p$ then $\Phi(p, t) = \Id$.
    \end{quote}

  \item The hyperbolicity condition is satisfied:
    \begin{equation}
      \label{eq:DiffeoValuedHypCondFlow}
      (2 r-1) (\rho_1 + \kappa \rho_0)  - \lambda \alpha  < 0. 
    \end{equation}
  \end{enumerate}
  Then there exists $\phi \in C^\alpha \bigl( M, \diff^{r-3}(N)
  \bigr)$ that solves
  \begin{equation}
    \label{eq:DiffeoFlowCoboundary}
    \Phi(x,t) = \phi(f^tx) \circ \phi^{-1}(x). 
  \end{equation}
\end{theorem}

\begin{rem}
  Using H\"older estimates it should be possible to show that the
  solution $\phi \in \diff^{r-\epsilon}(N)$~\cite{delaLlaveO99}. Using
  different arguments we hope to be able to show that $\phi \in
  \diff^{r}(N)$ so that there is no loss of differentiability. 
\end{rem}

\begin{proof}
  Let $x^* \in M$ be a point with a dense orbit,
  $\mathcal{O}(x^*)$. If we fix $\phi(x^*)$ then, by
  \eqref{eq:DiffeoFlowCoboundary}, we can define $\phi$ on all of
  $\mathcal{O}(x^*)$ by
  \begin{equation*}
    \phi(f^tx^*) = \Phi(x^*,t) \circ \phi(x^*).  
  \end{equation*}
  
  First, we show that the following H\"older condition on
  $\Phi$,
  \begin{multline}
    \label{eq:LocalHolderDiffeo}
    \text{ if $d_M( f^{t+T}x^*, f^tx^*) < \delta$ then}\\
    d_{r-2}\bigl( \Phi(f^tx^*,T), \Id \bigr) < K \, d_M( f^{t+T}x^*,
    f^tx^*)^\alpha
  \end{multline}
  implies the following H\"older condition
  \begin{multline}
    \label{eq:GlobalHolderDiffeo}
    \text{ if $d_M( f^{t+T}x^*, f^tx^*)<\delta$ then}\\
    d_{r-3}\bigl( \phi(f^{t+T}x^*), \phi(f^tx^*) \bigr) < K \, d_M(
    f^{t+T}x^*, f^tx^*)^\alpha.
  \end{multline}
  Condition \eqref{eq:GlobalHolderDiffeo} is precisely the condition
  that ensures that $\phi$ defined on $\mathcal{O}(x^*)$ can be extended
  to $\phi \in C^\alpha\bigl(M, \diff^{r-3}(N)\bigr)$.

  The collection $\{B(f^t x^*, \delta) \}_{t \in \R}$ is an open
  cover of $M$ and therefore by compactness we have a finite sub-cover
  $\{B(f^{t_i} x^*, \delta) \}_{i = 1}^m$. By finiteness there exists
  a constant $C>0$ such that $d_r( \phi(f^{t_i}x^*) , \Id) < C$ for $1
  \leq i \leq m$.  Given an arbitrary $t \in \R$ we choose $1 \leq i
  \leq m$ such that $d_M(f^t x^*, f^{t_i }x^*)< \delta$. From
  \eqref{eq:LocalHolderDiffeo} we get $d_{r-2} \bigl( \Phi(f^{t_i}x^*,
  t-t_i), \Id \bigr) < K \, \delta^\alpha$. By Lemma \ref{lem:MVT} we have
  \begin{multline*}
    d_{r-2}\bigl( \Phi(f^{t_i}x^*, t-t_i) \circ \phi(f^{t_i}x^*) ,
    \phi(f^{t_i}x^*) \bigr) \\
    \leq C \, d_{r-2} \bigl(  \Phi(f^{t_i}x^*,
    t-t_i), \Id \bigr) \leq C \, K \, \delta^\alpha.  
  \end{multline*}
  In particular $d_{r-2}\bigl( \phi(f^tx^*), \phi(f^{t_i}x^* \bigr)$
  is bounded. However since \linebreak $d_{r-2}\bigl(
  \phi(f^{t_i}x^*),\Id \bigr)$ is bounded we have that $d_{r-2}\bigl
  (\phi(f^tx^*), \Id \bigr)$ is bounded.

  Suppose that $t$ and $T$ are such that $d_M(f^tx*, f^{t+T}x^*) <
  \delta$. Assuming \eqref{eq:LocalHolderDiffeo} we then have $d_{r-2}
  \bigl( \Phi(f^{t}x^*, T), \Id \bigr)$ uniformly bounded.  We
  can again apply Lemma \ref{lem:MVT} to obtain
  \begin{equation*}
    d_{r-3} \bigl( \phi(f^tx^*),\phi(f^{t+T}x^*) \bigr) \leq C \,
      d_{r-3}  \bigl( \Phi(f^{t}x^*, T), \Id \bigr). 
  \end{equation*}
  Then applying \eqref{eq:LocalHolderDiffeo} we obtain
  \eqref{eq:GlobalHolderDiffeo}. Therefore to prove the theorem it
  suffices to prove \eqref{eq:LocalHolderDiffeo}.

  Suppose that $d_M( f^{t+T}x^*, f^tx^*)<\delta$ and apply the Anosov
  Closing Lemma, Lemma \ref{AnosovClosingFlow}, to obtain a periodic
  point $p \in M$ with $f^{T+\Delta}p=p$ and a point $z \in W^s(p)
  \cap W^u(f^tx^*)$. The periodic point satisfies:
  \begin{enumerate}
  \item $|\Delta| < K \, d_M( f^{t+T}x^*, f^tx^*)$.
    
  \item $d_M( f^{t+T} x^*, p) \leq K \, d_M( f^{t+T}x^*, f^tx^*)$.
    
  \end{enumerate}
  Again we let $C_F(s) := \Phi^{-1}(p,s) \circ \Phi(z,s)$. Here we
  encounter our first difficulty; even though $\Phi(p,s)$ and
  $\Phi(z,s)$ are differentiable as maps from $\R$ to $\diff^r(N)$ the
  map $C_F(s)$ is not since the group operation is not
  differentiable. The solution is to consider $C_F(s)$ in
  $\diff^{r-1}(N)$. The diffeomorphism $C_F(s)$ obeys the following
  differential equation in $T \diff^{r-1}(N)$,
  \begin{align*}
    \frac{d}{ds} C_F(s) &= D \Phi^{-1}(p,s)\circ \Phi(z,s) \cdot
    [\eta_{f^sz}\circ \Phi(z,s) - \eta_{f^sp}\circ \Phi(z,s)],\\
    C_F(0) &= \Id. 
  \end{align*}
  This equation is exactly analogous to the equation obtained in the
  Lie group case~\eqref{eq:CFdiffeq}. However, rather than deal with
  the Banach manifold $T \diff^{r-1}(N)$ we wish to use the familiar
  theory of differential equations on the compact manifold $N$.

  We wish to estimate $C_F(s)$ in $C^{r-1}$. Since $\Phi^{-1}(p,s) =
  \Phi(f^s p,-s)$ we can estimate $D \Phi^{-1}(p,s)$ and $\Phi(z,s)$
  using Lemma \ref{lem:DiffeomorphismGroup}. The fact that $\eta \in
  C^\alpha(M, \diff^r(N))$ means that we have
  \begin{equation}\label{eq:HolderEtaEstimate}
    \| D^n \eta_{f^s z} - D^n \eta_{f^sp}\| \leq C \, d_M( f^s z, f^s
    p)^\alpha
  \end{equation}
  for $0 \leq n \leq r$. Since $z \in W^s(p)$ we have 
  \begin{equation*}
    d_M( f^s z, f^s  p) < e^{-\lambda s} \, d_M(z,p).
  \end{equation*}
  From our statement of the Anosov Closing Lemma, Lemma
  \ref{AnosovClosingFlow}, we have
  \begin{equation*}
    d_M(z,p) \leq C d_M(f^tx^*,f^{t+T}x^*).
  \end{equation*}
  
  Using Lemma \ref{lem:DiffeomorphismGroup}, and
  \eqref{eq:HolderEtaEstimate}, we estimate
  \begin{align*}
    \| D \Phi^{-1}(p,s)\circ \Phi(z,s) \cdot
    [\eta_{f^s z} &\circ \Phi(z,s) - \eta_{f^sp} \circ \Phi(z,s)] \|\\
    &\leq \| D \Phi^{-1}(p,s) \|\, \|\eta_{f^s z}- \eta_{f^sp}\|\\
    &\leq C \, e^{(\rho_1 + \kappa \, \rho_0
 - \lambda \alpha) s} \,
    d_M(f^tx^*,f^{t+T}x^*)^\alpha\\
    \intertext{and}
    \| D \Phi^{-1}(z,s)\circ \Phi(p,s) \cdot
    [\eta_{f^s p} &\circ \Phi(p,s) - \eta_{f^s z} \circ \Phi(p,s)] \|\\
    &\leq \| D \Phi^{-1}(z,s) \|\, \|\eta_{f^s p}- \eta_{f^s z}\|\\
    &\leq C \, e^{(\rho_1 + \kappa \, \rho_0 - \lambda \alpha) s} \,
    d_M(f^tx^*,f^{t+T}x^*)^\alpha.
  \end{align*}
  Integrating these estimates, and using
  \eqref{eq:DiffeoValuedHypCondFlow}, we obtain
  \begin{equation*}
    d\bigl( C_F(s), \Id \bigr) \leq  \frac{C}{\lambda \alpha - \rho}
    d_M(f^tx^*,f^{t+T}x^*)^\alpha. 
  \end{equation*}

  It remains to estimate the derivatives $D^n C_F(s)$ and
  $D^nC_F^{-1}(s)$ for $n \leq r-1$.  Applying the Fa\`a di Bruno
  formula to the differential equation for $D^nC_F(s)$ we obtain
  \begin{align*}
    \frac{d}{ds} D^n C_F(s) &= D^n \Bigl( \bigl( D \Phi^{-1}(p,s)
    \cdot [\eta_{f^s z}-\eta_{f^sp}]\bigr) \circ \Phi(z,s) \Bigr)\\
    &=
    \begin{aligned}[t]
      \sum_{k_1,\dots,k_n} C_{k_1,\dots,k_n} &D^k \bigl( D
      \Phi^{-1}(p,s) \cdot [\eta_{f^s
        z}-\eta_{f^sp}]\bigr) \\
      &\cdot (D \Phi(z,s))^{\otimes k_1} \cdots (D^n
      \Phi(z,s))^{\otimes k_n}
    \end{aligned}
  \end{align*}
  The general term consists of products of two types of factors: 
  \begin{equation*}
    D^n \Phi(z,s)\text{ and }D^n \bigl(
    D \Phi^{-1}(p,s) \cdot [\eta_{f^s z}-\eta_{f^sp}]\bigr).
  \end{equation*}  
  The term $ D^n \Phi(z,s)$ is estimated using Lemma
  \ref{lem:DiffeomorphismGroup}.  To estimate 
  \begin{equation*}
    \| D^n \bigl[ D
    \Phi^{-1}(p,s) \cdot [\eta_{f^s z} -\eta_{f^sp}] \bigr] \|
  \end{equation*}
  we apply the Leibniz rule to obtain
  \begin{equation*}
    D^n \bigl[ D \Phi^{-1}(p,s) \cdot [\eta_{f^s z}
    -\eta_{f^sp}] \bigr] = \sum_{k=0}^n D^{k+1} \Phi^{-1}(p,s) \cdot
    [D^{n-k} \eta_{f^s z}-D^{n-k}\eta_{f^sp}].
  \end{equation*}
  As above, we obtain 
  \begin{align*}
    \|D^n \bigl[ D \Phi^{-1}(p,s) \cdot [\eta_{f^s z} &-\eta_{f^sp}]
    \bigr]\| \\
    &\leq C \, e^{(n+1) (\rho_1 + \kappa \, \rho_0) s} d_M(
    f^s z, f^s
    p)^\alpha \\
    &\leq C \, e^{(n+1) (\rho_1 + \kappa \, \rho_0) s - \lambda \alpha
      s} d_M (z,p)^\alpha\\ 
    &\leq C \, e^{(n+1) (\rho_1 + \kappa \, \rho_0) s - \lambda \alpha
      s} d_M (f^t x^*,f^{t+T}x^*)^\alpha
  \end{align*}
  
  Using the estimate from Lemma \ref{lem:DiffeomorphismGroup}, and our
  intermediate computation, we obtain 
  \begin{equation*}
    \bigl\|\frac{d}{ds}D^n C_F(s) \bigr\| \leq C \, e^{((2 n+1)
      (\rho_1 + \kappa \, \rho_0) - 
      \lambda \alpha ) s } \, d_M (f^t x^*,f^{t+T}x^*)^\alpha.
  \end{equation*}
  Similarly, we get 
  \begin{equation*}
    \bigl\|\frac{d}{ds}D^n C_F^{-1}(s) \bigr\| \leq C \, e^{((2 n+1)
      (\rho_1 + \kappa \, \rho_0) - \lambda \alpha ) s } \, d_M (f^t
    x^*,f^{t+T}x^*)^\alpha. 
  \end{equation*}
  Integrating, and  using \eqref{eq:DiffeoValuedHypCondFlow}, we obtain the
  following estimate
  \begin{equation*}
    d_{r-1}\bigl( C_F(s) , \Id \bigr) \leq \frac{C}{\lambda \alpha - (2 r -
      1) (\rho_1 + \kappa \, \rho_0)} \, d_M (f^t x^*,f^{t+T}x^*)^\alpha. 
  \end{equation*}
  Now we can write 
  \begin{align*}
    d_{r-1}\bigl( \Phi(z,T) , \Id \bigr) &\leq
    \begin{aligned}[t]
      d_{r-1} \bigl( \Phi(p,T) &\circ \Phi^{-1}(p,T) \circ \Phi(z,T) ,
      \Phi(p,T) \bigr)\\ 
      &+ d_{r-1} \bigl( \Phi(p,T), \Id \bigr)
    \end{aligned}
  \end{align*}
  From the cocycle property and the periodic orbit obstruction we have 
  $\Phi(p,T) = \Phi(p,-\Delta)$. From the Anosov Closing Lemma we have
  $|\Delta|< C \, d_M (f^{t}x^*, f^{t+T}x^*)$ and hence $\Delta$ is
  bounded. Thus we immediately get 
  \begin{equation*}
    d_{r-1}\bigl( \Phi(p,T) , \Id \bigr) \leq C \, |\Delta|. 
  \end{equation*}
  Consider $C_F(s)$ for $0 \leq s \leq T$ as a path. Using
  Lemma \ref{lem:MVT}, we estimate
  \begin{multline*}
    \ell_{r-1} \bigl( \Phi(p,T) \circ C_F(s) \bigr)\\ \leq C \,
    \|D\Phi(p,T)\|_{r-1} \, \bigl(1+\max_{s \in [0,T]}
    \|DC_F(s)\|_{r-2}\bigr)^{r-1} \, \ell_{r-1} \bigl( C_F(s) \bigr)
  \end{multline*}
  Since $\Phi(p,T+\Delta) = \Id$ and $\Delta$ is bounded we have
  $\|D\Phi(p,T)\|_{r-1}$ is uniformly bounded. Since $d_{r-1}\bigl(
  C_F(s) , \Id \bigr)$ is uniformly bounded we have by
  Lemma~\ref{lem:PathBounds} that $\max_{s \in [0,T]}
  \|DC_F(s)\|_{r-2}$ is uniformly bounded. Similarly, using Lemma
  \ref{lem:MVT}, we estimate
  \begin{align*}
    \ell_{r-1} \bigl( C_F^{-1} (s) \circ \Phi^{-1}(p,T) \bigr) \leq C
    \, (1+ \| D \Phi^{-1}(p,T) \|_{r-2})^{r-1} \ell_{r-1} \bigr( C_F^{-1}(s)
    \bigr). 
  \end{align*}
  The term $\| D \Phi^{-1}(p,T) \|_{r-2}$ is uniformly bounded since
  $\Delta$ is bounded.  Thus we finally obtain
  \begin{align*}
    d_{r-1} \bigl( \Phi(p,T) &\circ \Phi^{-1}(p,T) \circ \Phi(z,T) ,
    \Phi(p,T) \bigr)\\
    &\leq C \, d_{r-1} \bigl(C_F(T), \Id \bigr) \\
    &\leq \frac{C}{\lambda \alpha - (2 r -1) (\rho_1 + \kappa \,
      \rho_0)} \, d_M(f^tx^*, f^{t+T}x^*)^\alpha
  \end{align*}
  Combining these estimates, we get
  \begin{equation}
    \label{eq:PhizEstimateFlow}
    d_{r-1}\bigl( \Phi(z,T) , \Id \bigr) \leq C \, d_M(f^tx^*,
      f^{t+T}x^*)^\alpha 
  \end{equation}

  Now we compare the cocycle along orbits that converge in backward
  time. Let
  \begin{equation*}
    C_R(s) = \Phi^{-1}(f^{t+T}x^*,-s) \circ \Phi(f^Tz,-s). 
  \end{equation*}
  The function $C_R(s)$ satisfies the differential equation
  \begin{equation}\label{eq:CRDifferentialEquation}
    \frac{d}{ds} C_R(s) =\bigl( D\Phi^{-1}(f^{t+T}x^*,-s)
    \cdot [ \eta_{f^{t+T-s}x^*} - \eta_{f^{T-s}z}] \bigr)  \circ
    \Phi(f^Tz,-s).  
  \end{equation}
  Differentiating the differential equation
  \eqref{eq:CRDifferentialEquation} $n$ times, we obtain the
  differential equation 
  \begin{equation*} 
    \frac{d}{ds} D^n
    C_R(s) =D^n \Bigl( \bigl( D\Phi^{-1}(f^{t+T}x^*,-s) \cdot
    [\eta_{f^{t+T-s}x^*}-\eta_{f^{T-s}z}]\bigl) \circ \Phi(f^Tz,-s)
    \Bigl).
  \end{equation*}
  Now we apply the Fa\`a di Bruno formula to the differential equation
  for $D^nC_R(s)$ to obtain
  \begin{align*}
    \frac{d}{ds} D^n C_R(s) &= D^n \Bigl( \bigl( D
    \Phi^{-1}(f^{t+T}x^*,-s)
    \cdot [\eta_{f^{t+T-s} x^*}-\eta_{f^{T-s}z}]\bigr) \circ
    \Phi(f^Tz,-s) \Bigr)\\ 
    &=
    \begin{aligned}[t]
      \sum C_{k_1,\dots,k_n} & D^k \bigl( D \Phi^{-1}(f^{t+T}x^*,-s)
      \cdot [\eta_{f^{t+T-s} x^*}-\eta_{f^{T-s}z}]\bigr) \\
      &\cdot \bigl[ D \Phi(f^Tz,-s)^{\otimes k_1} \otimes \cdots
      \otimes D^n\Phi(f^Tz,-s)^{\otimes k_n} \bigr]
    \end{aligned}
  \end{align*}
  The general term in the sum consists of the product of two types of
  factors:
  \begin{equation*}
    D^k \bigl( D
    \Phi^{-1}(f^{t+T}x^*,-s) \cdot [\eta_{f^{t+T-s}
      x^*}-\eta_{f^{T-s}z}]\bigr)\text{ and }D^k\Phi(f^Tz,-s).
  \end{equation*}
  Taking the $n$-th derivative
  \begin{align*}
    D^n \bigl( &D\Phi^{-1}(f^{t+T}x^*,-s) \cdot
    [\eta_{f^{t+T-s}x^*}-\eta_{f^{T-s}z}] \bigr)\\
    &= \sum_{k=0}^n
    D^{k+1}\Phi^{-1}(f^{t+T}x^*,-s) \cdot [ D^{n-k}
    \eta_{f^{t+T-s}x^*}- D^{n-k}\eta_{f^{T-s}z}]
  \end{align*}
  which we may estimate by 
  \begin{align*}
    \bigl\|D^n \bigl( D\Phi^{-1}(f^{t+T}x^*,-s) \cdot &
    [\eta_{f^{t+T-s}x^*}-\eta_{f^{T-s}z}] \bigr) \bigr\|\\
    &\leq C \, e^{(n+1) (\rho_1 + \kappa \, \rho_0) s} d_M( f^{T-s} z,
    f^{t+T -s} x^*)^\alpha \\
    &\leq C \, e^{((n+1) (\rho_1 + \kappa \, \rho_0) - \lambda \alpha)
      s} \, d_M
    (f^Tz,f^{t+T}x^*)^\alpha\\
    &\leq C \, e^{((n+1) (\rho_1 + \kappa \, \rho_0) - \lambda \alpha)
      s} \, d_M (f^tx^* , f^{t+T}x^*)^\alpha
  \end{align*}
  Using the estimate from Lemma \ref{lem:DiffeomorphismGroup}, and our
  intermediate computation, we obtain 
  \begin{equation*}
    \bigl\|\frac{d}{ds}D^n C_R(s) \bigr\| \leq C e^{((2 n+1) (\rho_1
      + \kappa \, \rho_0) - \lambda \alpha ) s } \, d_M (f^t
    x^*,f^{t+T}x^*)^\alpha. 
  \end{equation*}
  Integrating, and using \eqref{eq:DiffeoValuedHypCondFlow}, we obtain 
  \begin{equation}
    \label{eq:DiffeomorphismDREstimate}
    d_{r-1}\bigl (C_R(T), \Id \bigr) \leq \frac{C} {\lambda \alpha -
      (2 r -1) 
      (\rho_1 + \kappa \, \rho_0)}\,   d_M (f^t x^*,f^{t+T}x^*)^\alpha.
  \end{equation}
  Finally we need to combine these estimates to obtain the final
  result. First observe that from the cocycle condition
  \begin{align*}
    C_R(T) &= \Phi^{-1}(f^{t+T}x^*,-T) \circ \Phi(f^{T}z,-T)\\
    &= \Phi(f^t x^*,T) \circ \Phi^{-1}(z, T) 
  \end{align*}
  Now
  \begin{align*}
    d_{r-1}\bigl( \Phi(f^t x^*,T), \Id \bigr) & \leq
    \begin{aligned}[t]
      d_{r-1} \bigl( \Phi(f^t x^*,T) &\circ \Phi^{-1}(z, T) \circ
      \Phi(z, T),\Phi(z, T) \bigr)\\
      &+ d_{r-1}\bigl( \Phi(z, T), \Id \bigr)
    \end{aligned}
  \end{align*}
  Using Lemma \ref{lem:MVT}, and \eqref{eq:PhizEstimateFlow}, gives
  \begin{align*}
    d_{r-2} \bigl( \Phi(f^t x^*,T) \circ \Phi^{-1}(z, T) \circ
    \Phi(z, T),\Phi(z, T) \bigr) \leq C \, d_{r-1}\bigl( C_R(T), \Id \bigr).
  \end{align*}
  Combining all our estimates yields
  \begin{equation}
    d_{r-2}\bigl( \Phi(f^t x^*,T), \Id \bigr) \leq C \,  d_M (f^t
     x^*,f^{t+T}x^*)^\alpha 
   \end{equation}
  which hence completes the proof.
\end{proof}

\subsection{Cocycles over an Anosov Diffeomorphism}

A statement analogous to Theorem \ref{thm:DiffeomorphismLivsic} holds
for diffeomorphism group valued cocycles over an Anosov
diffeomorphism. This can be obtained from the result on flows by
observing that the suspension of an Anosov diffeomorphism is an Anosov
flow. Proceeding in this fashion one needs to take a cocycle whose
generator is very close to the identity.

\begin{theorem}
  Let $M$ be a compact Riemannian manifold with $f:M \rightarrow M$
  be a $C^1$ topologically transitive $\lambda$-hyperbolic Anosov
  diffeomorphism. Let $N$ be a compact Riemannian manifold.

  Let $\Phi \in C^\alpha \bigr( M \times \Z, \diff^{r}(N) \bigr)$ and
  define $\eta(x) = \Phi(x,1)$. $\rho= \max_{x \in M} \|D \eta(x) \|$.

  Suppose that for the pair $f^t$ and $\eta$:
  \begin{enumerate}
  \item The periodic orbit obstruction vanishes:
    \begin{quote}
      If $f^np=p$ then $\Phi(p, n) = \Id$.
    \end{quote}

  \item The hyperbolicity condition is satisfied:
    \begin{equation}
      \label{eq:DiffeoValuedHypCondDiffeo}
      \rho^{2 r-1}  \lambda^\alpha  < 1. 
    \end{equation}
  \end{enumerate}
  Then there exists $\phi \in C^\alpha \bigl( M, \diff^{r-3}(N)
  \bigr)$ that solves
  \begin{equation}
    \label{eq:DiffeoValuedDiffeomorphismCoboundary}
    \Phi(x,n) = \phi(f^nx) \circ \phi^{-1}(x). 
  \end{equation}
\end{theorem}

\begin{proof}
  Let $x^* \in M$ be a point with a dense orbit,
  $\mathcal{O}(x^*)$. If we fix $\phi(x^*)$ then, by
  \eqref{eq:DiffeoValuedDiffeomorphismCoboundary}, we can define
  $\phi$ on all of $\mathcal{O}(x^*)$ by
  \begin{equation*}
    \phi(f^nx^*) = \Phi(x^*,n) \circ \phi(x^*).  
  \end{equation*}

  Exactly as in the flow case, we have that the following H\"older
  condition on $\Phi$,
  \begin{multline}
    \label{eq:DiffeoValuedLocalHolderDiffeo}
    \text{ if $d_M( f^{n+N}x^*, f^nx^*)<\delta$ then}\\
    d_{r-2}\bigl( \Phi(f^nx^*,N), \Id \bigr) < K \, d_M( f^{n+N}x^*,
    f^nx^*)^\alpha
  \end{multline}
  implies the H\"older condition 
  \begin{multline}
    \label{eq:DiffeoValuedGlobalHolderDiffeo}
    \text{ if $d_M( f^{n+N}x^*, f^nx^*)<\delta$ then}\\
    d_{r-3}\bigl( \phi(f^{n+N}x^*), \phi(f^nx^*) \bigr) < K \, d_M(
    f^{n+N}x^*, f^nx^*)^\alpha.
  \end{multline}
  Condition \eqref{eq:DiffeoValuedGlobalHolderDiffeo} means that
  $\phi$ can be extended to $\phi \in C^\alpha\bigl(M,
  \diff^{r-3}(N)\Bigr)$.

  In order to complete the proof it suffices to prove
  \eqref{eq:DiffeoValuedLocalHolderDiffeo}. Suppose that $d_M(
  f^{n+N}x^*, f^nx^*)<\delta$ and apply the Anosov Closing Lemma,
  Lemma \ref{AnosovClosingFlow}, to obtain a periodic point $p \in M$
  with $f^{N}p=p$ and a point $z \in W^s(p) \cap W^u(f^nx^*)$. The
  periodic point satisfies $d_M( f^{n+N} x^*, p) \leq K \, d_M(
  f^{n+N}x^*, f^nx^*)$ and the messenger point satisfies $d_M( f^{n+N}
  x^*, z) \leq K \, d_M( f^{n+N}x^*, f^nx^*)$. 

  Again we let $C_F(m) := \Phi^{-1}(p,m) \circ \Phi(z,m)$. The
  diffeomorphism obeys the following equation
  \begin{equation*}
    C_F(m+1) = \Phi^{-1}(p,m) \circ \eta^{-1}_{f^m p} \circ \eta_{f^m z}
    \circ \Phi(p,m). 
  \end{equation*}
  
  Let $p_s$ be a path joining $ \eta^{-1}_{f^m p} \circ \eta_{f^m z}$
  to $\Id$. Since we are only interested in paths that approach the
  optimal, and since $d_{r-1}(\eta^{-1}_x \circ \eta_{x'}, \Id)$ is
  uniformly bounded, we may assume that $\ell_{r-1}(p_s)$ and
  $\ell_{r-1}(p_s^{-1})$ are uniformly bounded. Now using Lemma
  \ref{lem:TechnicalMVT} we have  
  \begin{equation*}
    \ell_{r-1}( \Phi^{-1}(p,m) \circ p_s) \leq C \|D
    \Phi^{-1}(p,m)\|_{r-1} ( 1 + \max_{s \in [0,1]} \|D
    p_s\|_{r-2})^{r-1} \, \ell_{r-1}(p_s)
  \end{equation*}
  Using Lemma \ref{lem:PathBounds} and Lemma
  \ref{lem:PreliminaryDiffeo} we obtain
  \begin{equation*}
    \ell_{r-1}( \Phi^{-1}(p,m) \circ p_s) \leq C \, \rho^{ r\, m} \,
    \ell_{r-1}(p_s) 
  \end{equation*}
  where $C$ is independent of $m$. Applying Lemma \ref{lem:MVT} we get 
  \begin{multline*}
    \ell_{r-1} ( \Phi^{-1}(p,m) \circ p_s \circ \Phi(z,m) ) \\
    \leq C \,
    \rho^{ r\, m} \, \max_{k_1, \dots,k_{r-1}} \|D^1 \Phi(z,m)\|^{k_1}
    \cdots \|D^{r-1}\Phi(z,m)\|^{k_r} \, \ell_{r-1}( p_s)    
  \end{multline*}
  which, after applying Lemma \ref{lem:PreliminaryDiffeo}, yields
  \begin{equation*}
    \ell_{r-1} ( \Phi^{-1}(p,m) \circ p_s \circ \Phi(z,m) ) \leq C \,
    \rho^{r \, m} \, \rho^{(r-1)\,m} \, \ell_{r-1}(p_s).
  \end{equation*}
  By symmetry we get the same estimate for the inverse. Thus we have 
  \begin{equation*}
    d_{r-1}\bigl( C_F(m+1), C_F(m) \bigr) \leq C \, \rho^{(2r-1)\, m} \,
    d_{r-1}(\eta^{-1}_{f^m p} \circ \eta_{f^m z} , \Id ). 
  \end{equation*}

   Notice that by compactness there exists a $C >0$ so that
   $d_r(\eta_x, \Id) < C$ for all $x \in M$. Thus by Lemma
   \ref{lem:MVT} there exists a $K>1$ so that 
   \begin{align*}
     \frac{1}{K} \, d_{r-1} (\eta_x , \eta_{x'}) \leq d_{r-1} (\eta_x
     \circ \eta_{x'}^{-1} , \Id) \leq K \, d_{r-1} (\eta_x ,
     \eta_{x'}),\\
     \frac{1}{K} \, d_{r-1} (\eta_x , \eta_{x'}) \leq d_{r-1}
     (\eta_x^{-1} \circ \eta_{x'} , \Id) \leq K \, d_{r-1} (\eta_x ,
     \eta_{x'}).
   \end{align*}
  
   Since $\eta \in C^\alpha\bigl( M, \diff^r(N) \bigr)$ and $z \in
   W^s(p)$ we have
   \begin{equation*}
     d_{r-1}\bigl( C_F(m+1), C_F(m) \bigr) \leq C \, \rho^{(2r-1)\, m} \,
     \lambda^{\alpha\,m} \,  d_M( p , z )^\alpha.
   \end{equation*}
   In particular
   \begin{equation*}
     d_{r-1}\bigl( C_F(m+1), C_F(m) \bigr) \leq C \, \rho^{(2r-1)\, m} \,
     \lambda^{\alpha\,m} \,  d_M( f^{n+N}x^* , f^{n}x^* )^\alpha.
   \end{equation*}
   Thus for $m \geq 0$ we have the estimate
   \begin{equation*}
     d_{r-1} \bigl( C_F(m) , \Id \bigr) <
     \frac{C}{1-\rho^{2r-1}\lambda^\alpha} \, d_M( f^{n+N}x^* ,
     f^{n}x^* )^\alpha.   
   \end{equation*}
   Finally we observe that $C_F(N) = \Phi(z,N)$ since $\Phi(p, N)=
   \Id$ by the vanishing of the periodic orbit obstruction. Thus
   \begin{equation*}
     \label{eq:PhizEstimateDiffeo}
     d_{r-1}\bigl(\Phi(z,N), \Id \bigr)
     <\frac{C}{1-\rho^{2r-1}\lambda^\alpha} \, d_M( f^{n+N}x^* , 
     f^{n}x^* )^\alpha,  
   \end{equation*}
   and hence $\Phi(z,N)$ is uniformly bounded.

  Similar computations for $C_R(m) = \Phi^{-1}(f^{n+N}x^*, -m) \circ
  \Phi(f^Nz,-m)$ give the same result that for $m  \geq 0$
  \begin{equation*}
    \label{eq:CREstimateDiffeo}
    d_{r-1} \bigl( C_R(m) , \Id \bigr) < 
    \frac{C}{1-\rho^{2r-1}\lambda^\alpha} \, d_M( f^{n+N}x^* ,
    f^{n}x^* )^\alpha. 
  \end{equation*}
  Finally we observe that by the triangle inequality
  \begin{multline*}
    d_{r-1}\bigl( \Phi(f^n x^*, N) , \Id \bigr) \\
    \leq d_{r-1}\bigl( \Phi(f^nx^*,N) \circ \Phi^{-1}(z,N) \circ
    \Phi(z,N) , \Phi(z,N)
    \bigr)\\
    + d_{r-1}\bigl( \Phi(z,N), \Id \bigr).
  \end{multline*}
  Since $d_{r-1}\bigl( \Phi(z,N), \Id \bigr)$ is uniformly bounded by
  \eqref{eq:PhizEstimateDiffeo} we get 
  \begin{multline*}
    d_{r-2}\bigl( \Phi(f^n x^*, N) , \Id \bigr) \\
    \leq C \, d_{r-1}\bigl( \Phi(f^nx^*,N) \circ \Phi^{-1}(z,N) , \Id
    \bigr)\\
    + d_{r-2}\bigl( \Phi(z,N), \Id \bigr).
  \end{multline*}
  Observe that 
  \begin{equation*}
    \Phi(f^nx^*,N) = \Phi^{-1}(f^{n+N}x^*,-N) \text{ and }
    \Phi^{-1}(z,N) = \Phi(f^Nz,-N)
  \end{equation*}
  so
  \begin{equation*}
    d_{r-1}\bigl(
    \Phi(f^nx^*,N) \circ \Phi^{-1}(z,N) , \Id \bigr) = d_{r-1}\bigl(
    C_R(N) , \Id \bigr).
  \end{equation*}
  Combining \eqref{eq:PhizEstimateDiffeo} and
  \eqref{eq:CREstimateDiffeo} with our previous estimate we get 
  \begin{equation*}
    d_{r-2}\bigl( \Phi(f^n x^*, N) , \Id \bigr) <
    \frac{C}{1-\rho^{2r-1}\lambda^\alpha} \, d_M( f^{n+N}x^* , f^{n}x^* )^\alpha
  \end{equation*}
  which hence completes the proof. 
\end{proof}

An alternative proof by suspension is also possible though the
smallness conditions are much less explicit.  

\begin{proof}
  Let $\tilde{M}$ denote the usual suspension manifold 
  \begin{equation}
    \label{eq:11}
    \tilde M = \frac{M \times [0,1]}{ \sim} \qquad (x,1 ) \sim
    (f(x),0)  
  \end{equation}
  and define a flow $\tilde{f}^t:\tilde{M} \rightarrow \tilde{M}$ by
  $\tilde{f}^t(x,s) = (x,s+t)$. This flow is a $C^1$ topologically
  transitive $\lambda$-hyperbolic Anosov flow. It remains to show that
  we may select $\tilde{\eta}: \tilde{M} \rightarrow
  \mathfrak{X}^r(N)$ such that the cocycle $\tilde{\Phi}: \tilde{M}
  \times \R \rightarrow \diff^r(N)$ defined by
  \begin{equation*}
    \frac{d}{dt} \tilde\Phi\bigl((x,s),t\bigr) = \tilde\eta \bigl(
    \tilde{f}^t (x,s) \bigr) \circ
    \tilde\Phi\bigl((x,s),t\bigr), \qquad
    \tilde\Phi\bigl((x,s),0\bigr) = \Id
  \end{equation*}
  satisfies $\tilde\Phi\bigl((x,0),1\bigr) = \eta(x)$ and that the
  hyperbolicity condition for $\tilde\Phi$ is the same as for $\eta$.
  Provided $\eta(x)$ is sufficiently $C^1$ close to the identity then
  we can define $p : M \times [0,1] \rightarrow \diff^r(N)$ by
  \begin{equation*}
    p_{(x,s)} (y) = \exp_y \bigl[ m(s) \exp^{-1}_y \eta_x (y) \bigr]
  \end{equation*}
  where $m \in C^\infty\bigl([0,1],[0,1]\bigr)$ is $C^\infty$ flat at
  both $s=0$ and $s=1$, and has $0 \leq m'(s) \leq 1+\epsilon$. We
  have $p_{(x,0)}=\Id$ and $p_{(x_,1)}= \eta_x$. We may differentiate
  to obtain
  \begin{equation*}
    \frac{d}{ds} p_{(x,s)} = \tilde\eta_{(x,s)} \circ p_{(x,s)},
    \qquad  p_{(x,0)} = \Id
  \end{equation*}
  Now we can apply the flow version of the Liv\v{s}ic theorem to conclude
  that there exists $\tilde\phi \in C^\alpha\bigl( \tilde{M},
  \diff^{r-1}(N) \bigr)$ such that 
  \begin{equation*}
    \tilde\Phi\bigl((x,s),t\bigr) = \tilde\phi(\tilde{f}^t(x,s) \circ
    \tilde\phi(x,s). 
  \end{equation*}
  If we take $s=0$ and $t=n$ we obtain
  \begin{equation*}
    \tilde\Phi\bigl((x,0),n\bigr) = \tilde\phi(f^nx,0) \circ
    \tilde\phi(x,0). 
  \end{equation*}
  However we know that $\tilde\Phi\bigl((x,0),n\bigr)=\Phi(x,n)$ by
  construction and hence defining $\phi(x) = \tilde\phi(x,0)$ we obtain
  a solution to the coboundary equation. 
\end{proof}

\section{Existence of invariant conformal structures on the stable and
unstable bundles} \label{sec:conformal}

In this section we will consider possibility of defining metrics on
the stable and unstable bundles of Anosov systems that make the
mapping conformal.

Of course, the existence of expanding and contracting directions in an
Anosov map, makes it impossible to have metrics defined on the whole
tangent bundle which make the map conformal. The conformal structures
we consider in this section correspond to sub-Riemannian metrics on
the manifold, not to Riemannian ones.  In order to be able to do
analysis on the manifold, we will assume that the manifold is equipped
with a Riemannian metric, which we will assume analytic and which we
will refer to as {\sl background metric}.

Nevertheless, the existence of conformal metrics on the stable and
unstable bundles is a useful tool in the study of rigidity
questions. In \cite{delaLlave99, Llave04} it was shown that for
conformal Anosov systems, the only obstructions to smooth conjugacy were
the eigenvalues at periodic orbits (the paper above included an extra
assumption about the existence of global frames of reference in the
manifold, which we now remove).  One motivation for the papers above
was to understand geometrically the paper \cite{CastroM97}, which
studied related problems for analytic families on the torus.  The
papers \cite{KalininS03,Sadovskaya05} went further in the study of
geometric properties and showed that conformal Anosov systems are
smoothly equivalent to algebraic ones. Particularly interesting
systems of conformal Anosov systems are the geodesic flows on some
manifolds \cite{Yue}. For these systems, the results mentioned above
give a very strong rigidity of the manifolds.

Of course, conformal metrics have played an important role in the
theory of rigidity of manifolds.  Thus, the study in this section
provides a link between the dynamical rigidity of Anosov systems and
the geometric rigidity \cite{Mostow68, Mostow73}.  Indeed, once
the existence of a metric on the stable bundle that makes the mapping
conformal is established, the arguments of \cite{delaLlave99} are very
similar to those of \cite{Mostow68, Mostow73}. Indeed 
\cite{Jenkinson02} gave a proof of some particular cases of 
the results in \cite{Mostow68, Mostow73} using methods from 
the theory of differentiable rigidity.

The main result of this section will be Theorem \ref{conformal} which
gives necessary and sufficient conditions conditions for the existence
of families of conformal structures on the stable and unstable
bundles. The conditions involve the spectrum of the derivative of the
return maps at periodic orbits.

The proof of Theorem~\ref{conformal} that we will present is
remarkably similar to our proof of
Theorem~\ref{thm:LivsicLieGroupDiffeo}. We leave to the reader the
task of formulating the corresponding result on existence of conformal
metrics invariant under Anosov flows. The proof follows along
extremely similar lines.  We also note that similar arguments can be
used in the study of other geometric structures.

Putting together Theorem~\ref{conformal} and the global results of
\cite{KalininS03, Sadovskaya05}, we obtain that the global structure
of the manifold is determined by the eigenvalues at periodic
orbits. In the case of geodesic flows, it is well known that the
eigenvalues at periodic orbits are determined by the spectrum of the
Laplacian \cite{GuilleminK} (provided that the length spectrum is
simple).

We also note that in \cite{LlaveS05}, one can find the result that if
there is an invariant conformal structure that is in $L^p$ for $p$
sufficiently large, then there is a smooth conformal structure.

\subsection{Definitions and some elementary  results}
We start by briefly reviewing the theory of quasi-conformal maps and
setting the notation. All the results in this 
section are rather standard. Sources that we have found useful are
\cite{GheringP} and \cite{Vaisala71}.

Define the distortion of a differentiable 
map $f$ at a point $x$ with respect to a
metric $g$, denoted $K_g(f,x)$, by
\begin{equation}\label{distortion1} K_g(f,x) :=
  \frac{\max_{\substack{|v|=1\\ v\in T_xM}} |Df(x)v|_g}
  {\min_{\substack{|v|=1\\ v\in T_xM}} |Df(x)v|_g}
\end{equation} 
or, equivalently,
\begin{equation}
  \label{distortion2} 
  K_g(f,x) := \|Df(x)\| \cdot \|Df^{-1}\bigl(f(x) \bigr)\|
\end{equation} 
Of course $K_g(f,x) \geq 1$. We say that the map is conformal with
respect to $g$ if $K_g(f,x)=1$ for all $x$. Note that
\eqref{distortion2} makes it clear that
\begin{equation}
  \label{reciprocity} 
  K_{g} (f,x) = K_{g} (f^{-1},f(x)),
\end{equation}
and in particular, $f$ is conformal if and only if $f^{-1}$ is
conformal. 

If we take the sup and inf in \eqref{distortion1} 
when $v$ ranges over a sub-bundle $E$ of $TM$,  we
obtain the distortion along the  sub-bundle $E$,
which we will denote by $K_{g,E}(f,x)$.

Since any two metrics in a compact manifold are equivalent we have
for some constant $C_{g, \tilde g} > 0$
\begin{equation}\label{metricchange}
C_{g,\tilde g}^{-1} K_{\tilde g,E} (f,x)
 \le K_{g,E}(f,x) \le C_{g,\tilde g}
K_{\tilde g,E} (f,x)
\end{equation} where $C_{g,\tilde g}$ depends on the metrics $g$ and $\tilde
g$ but not on the map $f$ or the sub-bundle $E$.

The distortion and distortion along bundles satisfy a
sub-cocycle property
\begin{equation}\label{subcocycle} K_{g,E} (f_1\circ f_2,x) \le
K_{g, Df_2E}(f_1,f_2(x)) \, K_{g,E} (f_2,x).
\end{equation}
This follows from the chain rule and sub-multiplicativity of the
operator norm. 

We define $K_{g,E}(f) = \max_{x \in M} K_{g,E}(f,x)$. 

In particular, when $E$ is an invariant sub-bundle $Df(x)E_x \subset
E_{f(x)}$ we have:
\begin{equation}
  \label{subcocycle2} 
  K_{g,E} (f^n) \le [K_{g,E}(f)]^n.
\end{equation}
Hence, taking logarithms in \eqref{subcocycle} and using an elementary
sub-additive argument
\begin{equation*}
  \overline{K}_E(f) = \lim_{n \rightarrow \infty} \bigl( K_{g,E}(f) \bigr)^{1/n}
\end{equation*}
exists. By \eqref{metricchange},
$ \overline{K}_E(f)$ is independent of
the metric. It can be seen, but we will not use it here, that the 
distortion along a bundle is closely related to the spectral properties
of the weighted shift operator along the bundle. 

We also recall the following easy results about distortions of linear
operators on a fixed metric space, which we will use later to 
study the derivatives at fixed points.

Since $\|Av\|^2 = \langle v,A^*Av\rangle$, we have $\|A\|^ 2= \max
\spec (A^*A)$, $\|A^{-1}\|^{-1} = \min \spec (A^*A)$.  Hence
$$K(A)^2 = \frac{\max\spec (A^*A)}{\min\spec(A^*A)}\ .$$

\begin{prop}
  \label{orthogonal} 
  If $K(A) = 1$, then $\hat A = \dfrac{1}{\det (A)^{1/n}} A \in O(n)$, the
  orthogonal group corresponding to the metric. 
\end{prop}

\begin{proof}
  Since $K(A) =1$ there is a constant $C >0$ such that $\| A v \| = C
  \| v \|$ for all $v \in \R^n$.  Since $\det \hat A = 1$ we must have
  $C=1$. The desired result follows from the polarization argument.
\end{proof}

\begin{prop}
  \label{determinant} 
  \begin{equation}
    \label{determinantbound} 
    \|A\| \le |\det (A)|^{1/n}K(A)
  \end{equation}
\end{prop}

\begin{proof} 
  The desired result \eqref{determinantbound} is obvious for diagonal
  operators and, hence for diagonalizable, in particular symmetric
  operators.

  To prove the general case, we note that
  $$\|A\|^2 = \|AA^*\| \le \det (AA^*)^{1/n} K(AA^*) \le \det (A)^{2/n}K(A)^2$$

  Note also that the result is obvious for positive definite symmetric
  matrices since all the quantities in the formula can can be
  expressed in terms of the eigenvalues.

\end{proof}

\subsection{Results and their proofs}

In this subsection we formulate and prove Theorem~\ref{uniform} which
shows that the conformal properties of a transitive Anosov system are
determined by the behavior at periodic orbits.
Theorem~\ref{conformal} shows that, some spectral conditions on the
periodic orbits are enough to obtain the existence of invariant
conformal structures.

Similar theorems were proved in \cite{delaLlave99} under some extra
hypothesis about the manifold, in particular the existence of some
trivialization of the bundle.  The proof presented here is much more
geometric. Indeed, it is remarkably similar to the proof of
Theorem~\ref{thm:LivsicLieGroupDiffeo}. We will consider the
propagation of the structure and we will use the properties of the map
to show that the behavior at periodic orbits controls what happens on
a dense orbit.  In some auxiliary lemmas, to obtain the equivalence of
the hypothesis on the behavior at periodic orbits with other
hypothesis, we will need to use the specification property of
transitive Anosov systems.

\begin{theorem}\label{uniform}
  Let $M$ be a compact Riemannian manifold.  Let $f$ be a
  $C^{1+\alpha}$ $(0< \alpha \le \Lip)$ topologically transitive
  $\lambda$-hyperbolic Anosov diffeomorphism on $M$.  Let $E$ be a
  sub-bundle of the stable bundle $E^s$, which is invariant under $f$.
  
  Assume:
  \begin{itemize}
  \item[i)] There exists a constant $C_{\mathrm{per}}$ such that,
    whenever $f^N(x)=x$, with $N$ the minimal period of $x$, we have
    \begin{equation*}
      K_{g,E} (f^N,x) \le C_{\mathrm{per}}\ .
    \end{equation*}
  \item[ii)] $K_{g,E}(f)\le \rho$ with $\rho \, \lambda^\alpha < 1$.
  \end{itemize} 
  Then,  there exists $C>0$ such that 
for all $n \in \N$, we have
  \begin{equation*}
    K_{g,E} (f^n) \le  C.
  \end{equation*}
  Of course, an identical result holds for the unstable bundle.
\end{theorem}

\begin{rem} 
There is a version of Theorem~\ref{uniform} for flows. We leave 
the straightforward formulation as well as the proof to the reader. 

The proof of the result for flows requires only minor modifications of
the proof we present here. The required modifications can be read off
the corresponding modifications made to the proof of
Theorem~\ref{thm:LivsicLieGroupDiffeo} to get the proof of
Theorem~\ref{thm:LivsicLieGroupFlow}. The only difference lies in the
version of Anosov Closing Lemma that we use and the fact that for
flows we have a term corresponding to the change in period to control.
\end{rem}

\begin{rem} 
  Note that because the distortion for a metric is equivalent to the
  distortion for another (see \eqref{metricchange}) if the hypothesis
  i) or the conclusions hold for one metric they hold for any other
  metric.  Hypothesis ii) on the other hand, depends on the metric
  $g$.  Later, when it is better motivated by the proof, we will
  introduce a replacement hypothesis \eqref{cocyclegrowth}, which,
  though harder to state, is more geometrically natural. We note that
  this hypothesis is a close analogue of the localization estimates of
  the first part of this paper, since it can be interpreted as a
  spectral condition for a transfer operator.
\end{rem}

\begin{proof} 
  We will show that for all $n\in\N$, we have $K_{g,E^s}(f^n) \leq C$.
  We recall that by the specification property of transitive Anosov
  systems~\cite[Theorem 18.3.12]{KatokH95}, given $\epsilon>0$
  sufficiently small we can find $L \in \N $ such that for every $x\in
  M$ and $n\in\N$, there exists a period point $p$ with minimal period
  $n + m$ with $0 \leq m \leq L$ that satisfies
  \begin{equation*}
    d_M\bigl(f^i(x),f^i(p) \bigr) \leq \epsilon \qquad 0\le i\le n\ .
  \end{equation*}
  We will choose one such $\epsilon > 0$ that will remain fixed for the
  rest of the proof.  This $\epsilon$ will have to satisfy a finite
  number of smallness conditions which we will make explicit when we
  need them. 

  By the local product structure~\cite[Proposition 6.4.21]{KatokH95}
  we have
  \begin{equation*}
    W_{\loc}^s(p) \cap W_{\loc}^u(x) =\{ z \}
  \end{equation*}
  with $d(x,z)<\epsilon$ and $d(p,z)<\epsilon$. We may suppose that 
 \begin{equation*}
    d_M\bigl(f^iz,f^ix \bigr) \leq \epsilon \qquad 0 \leq i \leq n.
  \end{equation*}
  
  We note that
  \begin{equation}
    \label{distortionp}
    \begin{array}{rcl} 
      K_{g,E^s} (f^n,p) &\le &K_{g,E^s}
      (f^{n+m},p)\, K_{g,E^s}(f^{-m},p) \\ \noalign{\vskip6pt} &\le &
      C_{\mathrm{per}} \, K_{g,E^s} (f^{-m})
    \end{array}
  \end{equation}
  with $C_{\mathrm{per}}$ the constant in assumption i). Hence, it
  suffices to argue that it is possible to control $K_{g,E^s} (f^n,x)$
  in terms of $K_{g,E^s}(f^n,p)$.

  \subsubsection{Some local coordinates}
  We will find it convenient to use matrix notation so we introduce a
  coordinate patch about each point of the orbit of $x$. It is
  important to note that these coordinate patches do not need to agree
  in the regions where they overlap. Hence, they do not impose any
  restriction on the manifold $M$.  Similar constructions happen in
  \cite{HirschPPS69}.

  One convenient way of choosing these coordinate systems is picking a
  coordinate system $\psi_i$ on $T_{f^i(x)}M$ with
  \begin{equation*}
    \langle u,v \rangle_{g(f^i(x))} = \langle \psi_i u, \psi_i v \rangle_2 
  \end{equation*}
  and then setting
  $$\Psi_i(y) = \psi_i \circ \exp_{f^i(x)}^{-1} (y)$$ 
  where the domain $U_i$ chosen as balls of radius $1/2$ the
  injectivity radius of the metric. We neither assume, nor require,
  that the $U_i$ are disjoint. Notice, however that these
  coordinate patches, centered around  each point in the orbit 
 include balls of radius bounded uniformly from
 below. Furthermore, the coordinate functions are  $C^r$ diffeomorphisms 
and they are uniformly $C^r$. So that, to  show that a geometric 
object is $C^r$ it suffices to show that its coordinates representations 
are uniformly $C^r$. Furthermore, the $C^r$ norm of a geometric object 
will be equivalent to the supremum of the $C^r$ norms of the coordinate
representations. 

Once we choose a system of coordinates, we can identify $Df(y)$, for
$y\in U_i\cap f^{-1}(U_{i+1})$, with the matrix
\begin{equation}\label{identified} 
  \eta_i(y) = D\Psi_{i+1}(f(y)) \, Df(y) \, (D\Psi_i(y))^{-1}.
\end{equation}
We have chosen the notation $\eta_i$ by analogy with the proof of
Theorem~\ref{thm:LivsicLieGroupDiffeo}.

We can now proceed in a way very similar to the proof of
Theorem~\ref{thm:LivsicLieGroupDiffeo}. For $y \in U_0$ we define
\begin{equation*}
  \Phi(y,m) = \eta_{m-1}(f^{m-1} y) \cdots \eta_0(y), \qquad m \geq 1 
\end{equation*}
and for $y \in U_n$ we define
\begin{equation*}
  \Phi(y,-m) = \eta_{n-m}(f^{-m} y) \cdots \eta_{n-1}(f^{-1}y), 
  \qquad m \geq 1 
\end{equation*}
where these products are defined. Our goal is to show $K \bigl(
\Phi(x,n) \bigr) < C $. By compactness, and our choice of
geometrically natural coordinate systems, showing that the distortion
of the coordinate representation of the derivative cocycle is
uniformly bounded suffices to show that the distortion of the
derivative cocycle itself is uniformly bounded.
  
If $y, z$ are points whose orbits are converging in forward time so
that $f^i(z)$ is always in the coordinate neighborhood of $f^i(y)$, we
can use the same coordinate patch.

Define $C_F(m) = \Phi^{-1}(p,m) \Phi(z,m)$

\begin{rem}
  The point of introducing coordinates is to avoid unnecessary
  complication with connections.  Intrinsically $C_F(m): T_zM
  \rightarrow T_pM$ defined by $Df^{-m}(f^m(p)) S Df^m(z)$ where $S$
  is an identification between $T_{f^m(z)}M$ and $T_{f^m(p)}M$. Note
  that, because the orbits of $z$ and $p$ are converging, we can
  always define the comparisons.  Intuitively, as the points converge,
  the identifications become less important. Using the identifications
  between the end points is a possible alternative setup. This is what
  are called {\sl connectors} in \cite{HirschPPS69}.
\end{rem}

We have the recurrence:
\begin{equation} \label{Crecurrence2}
    \begin{aligned}
      C_F(m+1) &= \Phi^{-1}(p,m) \eta^{-1}(f^m p) \eta(f^m z) \Phi(z,m)\\
      &=
      \begin{aligned}[t]
        \Phi^{-1}(p,m) &\bigl[ \eta^{-1}(f^m p) \eta(f^m z) -\Id
        \bigr]
        \Phi(z,m)\\
        &+ C_F(m)
      \end{aligned}
    \end{aligned}
  \end{equation}

  Of course, we are interested only in the distortion of $C_F(m)$, so,
  we normalize the matrices to have determinant $1$ with respect to
  the background metric.
  \begin{align*}
    \hat\eta_i(y) &= \frac{1}{(\det \eta_i(y))^{\frac{1}{n}}} \eta_i
    (y)\\
    V_i(y) &= (\det \eta_i(y))^{\frac{1}{n}}.  
  \end{align*}
  For $\epsilon > 0$ sufficiently small we can choose $\tilde \rho >0$
  and $\sigma >0$ such that $\sigma \tilde\rho^2 \lambda ^\alpha < 1$
  and 
  \begin{align*}
    K \bigl( \eta_i(y) \bigr)&\leq \tilde\rho && \text{for $d_M(y, f^i
      x) < \epsilon$,}\\ 
    \frac{V_i(y_1)}{V_i(y_2)} &\leq \sigma && \text{for $d_M(y_1, f^i
      x) < \epsilon$ and $d_M(y_2, f^i x) <  \epsilon$}. 
  \end{align*}
  
  The recurrence \eqref{Crecurrence2} can be written as:
  \begin{multline}\label{Crecurrence3}
    C_F(m+1) = \frac{V_{m}(f^{m}z)}{V_{m}(f^{m}p)} \cdots
    \frac{V_{0}(z)}{V_{0}(p)}\\
    \cdot\hat\Phi^{-1}(p,m) \bigl( \hat\eta^{-1}(f^mp) \hat\eta(f^mz)
    - \Id \bigr) \hat\Phi(z,m) + C_F(m).
  \end{multline}
  Since $f\in C^{1+\alpha}$ we have $\hat\eta_i \in C^\alpha$. Thus,
  since $z \in W^{s}(p)$
  \begin{equation*}
    \|\hat\eta^{-1}(f^mp) \hat\eta(f^m z) - \Id \| \le C_2
    (\lambda^{\alpha})^m.
  \end{equation*}
  Thus, by Proposition \ref{determinant},  
  \begin{equation*}
    \|C_F(m+1)-C_F(m)\| < C_2 \bigl( \sigma \tilde\rho^2 \lambda^\alpha
    \bigr)^m  .
  \end{equation*}
  Thus $\|C_F(m)\|$ is uniformly bounded. Similarly we obtain
  $\|C_F^{-1}(m)\|$ is uniformly bounded. Thus 
  \begin{equation}\label{eq:ForwardConformalEstimates}
    \begin{aligned}
      \| \Phi(z,n) \| &\leq \| \Phi(p,n) \| \, \| C_F(n) \|, \\
      \|\Phi^{-1}(z,n)\| &\leq \| \Phi^{-1}(p,n) \| \, \| C_F^{-1} (n)
      \|.
    \end{aligned}
  \end{equation}
  
  Now we perform the same computations along orbits converging
  exponentially in backwards time.  Let $C_R(m) = \Phi^{-1}(f^n x, -m)
  \Phi(f^n z, -m)$ and perform the same computations to obtain
  $\|C_R(m)\|$ and $\|C_R^{-1}(m)\|$ are uniformly bounded. Finally
  observe that we have the pseudo-cocycle property
  \begin{align*}
    \Phi^{-1}(f^n x, -n) &= \Phi(x,n)\\
    \Phi(f^nz,-n)&= \Phi^{-1}(z,n) 
  \end{align*}
  and so 
  \begin{equation}\label{eq:BackwardConformalEstimates}
    \begin{aligned}
      \| \Phi(x,n) \| &\leq \| C_R(n) \| \, \|\Phi(z,n)\|\\
      \| \Phi^{-1}(x,n) \| &\leq \| \Phi^{-1}(z,n) \| \,
      \|C_R^{-1}(n)\|.
    \end{aligned}
  \end{equation}
  Replacing the norms $ \|\Phi(z,n)\|$ and $ \| \Phi^{-1}(z,n) \|$ in
  \eqref{eq:BackwardConformalEstimates} with their estimates from
  \eqref{eq:ForwardConformalEstimates} and taking the product we  obtain
  \begin{align}\label{eq:3}
    K\bigl( \Phi(x,n) \bigr) \leq K\bigl (C_F(n) \bigr) \, K\bigl
    (C_R(n) \bigr) \, K\bigl ( \Phi(p,n) \bigr).
  \end{align}
  We have shown that 
  $ K\bigl (C_F(n) \bigr)$ and $K \bigl (C_R(n) \bigr)$ 
  are uniformly bounded. By our remarks at the outset we have
  $K\bigl ( \Phi(p,n) \bigr) $ 
  is uniformly bounded due to assumption  i). 
\end{proof}

\begin{rem} We call attention to the fact that the only place in the
  proof of Theorem~\ref{uniform} where we use the hypothesis ii) is in
  the estimates of \eqref{Crecurrence3}.  What we actually need is that
  the norms of $\Delta_F^m$ and $\Delta_R^m$ defined by
  \begin{align*}
    \Delta_F^m(A) &:=  \hat \Phi^{-1}(p,m) \, A \, \hat\Phi(z,m),\\
    \Delta_R^m(A) &:=  \hat \Phi^{-1}(f^nx,-m) \, A \, \hat \Phi(f^nz,-m)
  \end{align*}
  satisfy
  \begin{equation} \label{cocyclegrowth} \|\Delta_R^m\|,
    \|\Delta_F^m\| \leq C \, (\sigma \lambda^\alpha)^{-m}
  \end{equation} 
  for $n$ large enough.
  
  Note that the condition \eqref{cocyclegrowth} on the asymptotic
  growth of the of cocycles is independent of the background metric.
  It can be used in place of assumption ii) in
  Theorem~\ref{uniform}. Note the analogy with the localization
  estimates in the first part of this paper. In subsequent results
  (Theorem~\ref{conformal}) we will also have similar hypothesis.
\end{rem}

Now we come to the second main result in this section, which 
characterizes the existence of conformal structures by the behavior 
at periodic orbits. 

\begin{theorem}\label{conformal} 
  Let $M$ be a compact $d$-dimensional Riemannian manifold endowed
  with a Riemanian metric $g$.  (We refer to such a metric as the
  background metric.)  Let $f$ be a $C^{1+\delta}$ topologically
  transitive Anosov diffeomorphism.  $(0<\delta \le \Lip)$.
  \begin{itemize}
  \item[{\rm i)}] Assume whenever $f^N(x)=x$, then
    $$Df^N|_{E^s} = \gamma_{S,N}(x) \Id$$
    for some real numbers $\gamma_{S,N} (x)$
  \item[{\rm ii)}] Assume that
    $$K_{g,E^s} (f) \le a$$ 
    with $a$ sufficiently close to $1$.
  \end{itemize} Then there exists a $C^\delta$ metric $g^s$ on $E^s$
  such that $f$ is conformal on the stable leaves with respect to $g^s$.
  
  Analogous results hold also for unstable bundles and for Anosov
  flows.
\end{theorem}

Of course the metrics are highly non-unique since we can multiply by
an arbitrary function.

\begin{rem} 
  Note that hypothesis i) does not depend on the background metric but
  hypothesis ii) does.

  As mentioned before, later we will formulate a geometrically natural
  -- but somewhat harder to state -- hypothesis, \eqref{asympgrowth}
  which can be used in place of ii).
\end{rem}

\begin{rem} 
  Note that if we fix the conformal structure at one point $x^*$, the
  fact that $f$ is conformal, determines it at $f x^*$.  If we choose
  $x^*$ so that its orbit is dense, the conformal structure at $x^*$
  determines it in the whole manifold. Hence, there is at most a
  finite dimensional family of invariant conformal structures.  The
  proof of Theorem~\ref{conformal} is done by choosing a conformal
  structure at $x^*$, propagating the conformal structure along the
  dense orbit of $x^*$, and then, using the hypothesis on the spectrum
  of the periodic orbits showing it  extends to the whole manifold.

  Note that then, we prove that, under the hypothesis on periodic
  orbits, we get that there is a family of invariant conformal
  structures with the dimension of the space of conformal
  structures at one point. 
 Conversely, if there is family of conformal structures
invariant under the map whose dimension is the 
dimension of conformal structures at one point, then the derivative at
  a fixed point has to be the a multiple of the identity.
\end{rem}

\begin{rem} 
  A result very similar to Theorem~\ref{conformal} was proved in
  \cite{delaLlave99} but the proof required the existence of a global
  frame in the manifold.  It was shown in \cite{delaLlave99} that if
  the map $f$ is $C^r$, continuous invariant conformal structures are
  actually $C^{r-1 - \epsilon}$, $r \in \N \cup \{\infty, \omega\}$.
  The proof of the bootstrap of regularity in \cite{delaLlave99} is
  very geometric and works without extra assumptions on the existence
  of frames. So, we will just refer to that paper.  In
  \cite{LlaveS05}, it was shown that if an invariant conformal
  structures is in $L^p$ for $p$ large enough, then it is continuous
  and, therefore differentiable.

  The papers \cite{KalininS03, Sadovskaya05, KalininS07} show that the
  existence of a conformal structure on the stable and the unstable
  foliations for an Anosov systems, implies also some global
  properties of the manifold.
\end{rem}

\begin{proof}
Let $x^*$ be a point with a dense orbit.  We will define the desired
metric along the orbit of $x^*$ and show it extends to the whole
manifold in a H\"older fashion.

We consider the bundle isomorphism $f_\#$ on the bundle of quadratic
forms on $E^s$. Denoting the space of quadratic forms on $E_x^s$ by
$Q_x$ we define $f_{\#} : Q_x \to Q_{f(x)}$ by
$$f_{\#} g = (\det Df_s (x))^{2/d} g (Df_s^{-1} (f(x)))^{\otimes 2}$$ 
where $Df_s$ denotes the derivative of $f$ restricted to $E^s$ and
$d$ is the dimension of the stable bundle.  The determinant is
measured with respect to the background metric $g$.

We note that we can use the background metric $g$ to measure the
norm of $f_{\#}$.

We claim that, by assuming that $K_{g,E^s}(f)$ is sufficiently close
to one, we can ensure that $\|f_{\#}\|$ is as close to 1 as we want.
Hence, using assumption ii) we can assume in the proof that
$\|f_{\#}\|$ is sufficiently close to 1.

Indeed, if we choose coordinates in $E_x^s$, $E_{f(x)}^s$ in such a
way that $g_x$, $g_{f(x)}$ are represented by the identity matrix,
the operator $f_{\#}$ reduces to the operator
$$\mathcal{L} (S) = A^t SA/\det (A)^{2/n}$$ 
acting on the space of symmetric matrices, where $A$ is the coordinate
representation of $Df^{-1}(f(x))$.
  
Applying Proposition~\ref{determinant} we obtain
$\|A /|\det  A|^{1/d}\| \le K(A)$ from which the claim follows.
  
The hypothesis that we will need in the rest of the argument is
\begin{equation}\label{asympgrowth} 
\mbox{ii$'$)}\qquad \|f_{\#}^\ell
    \| \le C\mu^n - 2\ell
\end{equation} with $\mu$ smaller than $(\lambda^\delta)$ --- with
$\lambda$ the hyperbolicity exponent.
  
The rest of the proof is very similar to the proof of
Theorem~\ref{thm:LivsicLieGroupDiffeo} and Theorem~\ref{uniform}.  We
pick a metric on $g^s_{x^*}$ on $E_{x^*}^s$ and define 
\begin{equation*}
g^s_{f^nx^*} := f_{\#}^n g^s_{x^*}. 
\end{equation*}
To check that the metric $g^s$ defined along the orbit of $x^*$ extends
in a H\"older fashion to the whole of $M$ we recall that by the Anosov
Closing Lemma, Lemma \ref{AnosovClosingDiffeomorphism}, there
exists $\epsilon > 0$ such that if $d(f^n x^*, f^{n+N} x^*) \leq
\epsilon$, then there is a periodic point $p$ with $f^Np=p$ such
that
\begin{equation*}
  d(f^{n+i} x^* ,f^ip) \leq \epsilon
\end{equation*}
We will estimate $f_{\#}^N$ restricted to  $Q_{f^nx^*}$ using that, by
hypothesis i),  $f_{\#}^N$ restricted to $Q_p$ .  The estimates will depend
only on $\epsilon$ but will be uniform in $N$.
  
From the Anosov Closing Lemma, Lemma
\ref{AnosovClosingDiffeomorphism} we obtain a ``messenger'' point
$z\in W_{\mathrm{loc}}^{u}\bigl(f^n(x^*)\bigr) \cap
W_{\mathrm{loc}}^{s}(p)$.  We can take local coordinate systems
defined on neighborhoods $U_i$ around $f^{n+i}(x^*)$ for $0\le i\le N$,
in such a way that $f^i(p)$ is contained in the coordinate patch
$U_i$.  
  
We realize that $f_{\#}$ is $C^\alpha$ in the whole manifold.  Denote
by $\eta_i(y)$ the coordinate representation of $f_{\#}$ acting on
$Q_y$ for $y \in U_i$. Showing that the metric $f^{n+N}_\# g_{x^*}$ on
$E^s_{f^{n+N}x^*}$ is close enough to the metric $f^{n}_\# g_{x^*}$ on
$E^s_{f^{n}x^*}$ reduces to estimating
\[
[\eta(f^N(q))\cdots \eta(a)]^{-1}\eta (f^{n+N} x^*) 
  \cdots \eta(f^n(x^*)) -\Id
\]

As in the previous results we proceed to estimate
\begin{gather*} [\eta (f^N(z))\cdots \eta (z)]^{-1}\eta
    (f^{n+N}(x^*))\cdots \eta (f^n x^*) -\Id\\ [\eta (f^N(q))\cdots \eta
    (q)]^{-1}\eta (f^N(z))\cdots \eta(z) -\Id
\end{gather*}
  
The proof is exactly the same as in
Theorem~\ref{thm:LivsicLieGroupDiffeo} and we refer to the proof of
this result for the details.
\end{proof}

\bibliographystyle{alpha}
\bibliography{livsic}

\end{document}